\documentclass[a4paper]{amsart}

\usepackage{verbatim,amsmath,amssymb}
\usepackage[T1]{fontenc}
\usepackage[latin1]{inputenc}
\usepackage{graphicx}

\newtheorem{thm}{Theorem}
\newtheorem{prop}[thm]{Proposition}
\newtheorem{lem}[thm]{Lemma}
\newtheorem{cor}[thm]{Corollary}
\newtheorem{ex}[thm]{Example}
\newtheorem{rem}[thm]{Remark}
\newtheorem{definition}[thm]{Definition}
\newtheorem{conjecture}[thm]{Conjecture}

\numberwithin{equation}{section} \numberwithin{thm}{section}

\newcommand{\cal}{\mathcal}

\newcommand{\dist}{{\rm dist}\,}

\newcommand{\R}{{\mathbb R}}
\newcommand{\N}{{\mathbb N}}

\newcommand{\nor}{{\rm nor}\,}
\newcommand{\supp}{{\rm supp}\,}
\newcommand{\diam}{{\rm diam}\,}
\newcommand{\Nor}{{\rm Nor}\,}
\newcommand{\Tan}{{\rm Tan}\,}

\newcommand{\Ha}{{\cal H}}
\newcommand{\bd}{\partial}

\newcommand{\ind}[1]{\mathbf{1}_{#1}}
\newcommand{\eps}{\varepsilon}
\newcommand{\sL}{\cal{L}}

\newcommand{\sT}{\cal{T}}

\newcommand{\sC}{\cal{C}}

\newcommand{\udim}{\overline{\dim}}
\newcommand{\var}{{\rm var}}

\newcommand{\gC}{\widetilde{C}}
\newcommand{\esslim}[1]{\underset{#1}{\mathrm{esslim}\,}}
\newcommand{\esslimsup}[1]{\underset{#1}{\mathrm{esslimsup}\,}}

\newcommand{\esssup}[1]{\underset{#1}{\mathrm{esssup}\,}}

\def\en{\mathbb N}
\def\er{\mathbb R}

\def\H{\mathcal H}

\def\Q{\mathcal Q}

\newcommand{\mydot}{\,\cdot\,}

\newcommand{\card}{\operatorname{card}}
\newcommand{\co}{\operatorname{co}}
\newcommand{\conv}{\operatorname{conv}}

\newcommand{\aff}{\operatorname{aff}}

\newcommand{\cl}[1]{\overline{#1}}

\begin{document}
\title{Scaling exponents of curvature measures}
\author{Du\v{s}an Pokorn\'y}
\address{Charles University, Faculty of Mathematics and Physics, Sokolovsk\'a
83, 18675 Praha~8, Czech Republic}
\author{Steffen Winter}
\address{Karlsruhe Institute of Technology, Department of Mathematics, Kaiserstr.\ 89, 76133 Karlsruhe, Germany}

\date{\today}

\begin{abstract}
Fractal curvatures of a compact set $F\subset\R^d$ are roughly defined as suitably rescaled limits of the total curvatures of its parallel sets $F_\eps$ as $\eps$ tends to $0$ and have been studied in the last years in particular for self-similar and self-conformal sets.  This previous work was focussed on establishing the existence of (averaged) fractal curvatures and related fractal curvature measures in the generic case when the $k$-th curvature measure $C_k(F_\eps,\cdot)$ scales like $\eps^{k-D}$,
where $D$ ist the Minkowski dimension of $F$. In the present paper we study the nongeneric situation when the scaling exponents are not determined by the dimension of $F$. We demonstrate that the possibilities for nongeneric behaviour are rather limited and introduce the notion of local flatness, which allows a geometric characterization of nongenericy in $\R$ and $\R^2$. We expect local flatness to be characteristic also in higher dimensions.
The results enlighten the geometric meaning of the scaling exponents.
\end{abstract}

\maketitle

\section{Introduction}
Curvature measures are important geometric tools in fields such as convex geometry, differential geometry, integral geometry and geometric measure theory. They have been defined and studied for various different set classes such as convex sets and their unions, differentiable manifolds, sets with positive reach \cite{F59}, subanalytic sets \cite{Fu94} etc. In \cite{W1}, a certain extension to fractal sets $F\subset\R^d$ has been suggested by means of the approximation of $F$ by its parallel sets:
For a bounded set $F\subset\R^d$ and $\eps\ge 0$, denote by
\begin{equation*}
F_\eps =\{y\in\R^d: \dist(y,F)\le\eps\}
\end{equation*}
the $\eps$-parallel set of $F$. Suppose the curvature measures $C_0(F_\eps,\cdot),\ldots,$ $C_d(F_\eps,\cdot)$ of $F_\eps$ are well defined for almost all $\eps>0$ in the sense of Rataj and Zähle \cite{RZ05}, see more details below in Section~\ref{sec:pre}.
In this list of (in general signed) measures, the surface area $C_{d-1}(F_\eps,\cdot)=\frac 12 \Ha^{d-1}(\bd F_\eps\cap\cdot)$ and the volume measure $C_d(F_\eps,\cdot)=\lambda_d(F_\eps\cap\cdot)$ are included.
Denote the total masses of these measures by $C_k(F_\eps):=C_k(F_\eps,\R^d)$, $k=0,\ldots,d$. Then, for $s\geq 0$, the \emph{($s$-dimensional) $k$-th fractal curvature} of $F$ is
defined by
\begin{equation}\label{eq:fc}
\sC^s_k(F):=\esslim{\eps\to 0} \eps^{s-k} C_k(F_\eps),
\end{equation}
or, more generally, by
\begin{equation}\label{eq:afc}
\sC^s_k(F):=\lim_{\delta\to 0} \frac 1{|\ln\delta|}\int_\delta^1\eps^{s-k-1} C_k(F_\eps) d\eps,
\end{equation}
provided these limits exist (possibly being $+\infty$ or $-\infty$). It is clear that, if the essential limit in \eqref{eq:fc} exists, then the limit in \eqref{eq:afc} exists as well and both values coincide, justifying to speak of \eqref{eq:afc} as a generalization of \eqref{eq:fc}. In the literature, also the term \emph{average fractal curvatures} is used for the limits in \eqref{eq:afc}. The exponent $s$ has to be chosen appropriately.  Typically $s=D:=\dim_M F$ is the right choice for all $k$ and therefore, up to now, fractal curvatures have mainly been studied with this choice for the scaling exponents.

Indeed, the fractal curvatures $\sC_k^D(K)$ have first been considered in \cite{W1} for self-similar sets $K\subset\R^d$ satisfying the open set condition (OSC) under the additional assumption that the parallel sets $K_\eps$ are polyconvex. For nonlattice self-similar sets, the existence of the limit \eqref{eq:fc} (in fact not only as an essential limit but as an ordinary limit) was shown, while for lattice self-similar sets only the existence of the average limit \eqref{eq:afc} has been established. This fundamental difference between lattice and nonlattice self-similar sets had been observed before, in particular for Minkowski contents in \cite{lap93,fal95} (for $d=1$) and \cite{gatz} (for general $d$). The assumption of polyconvexity has been replaced by different weaker (but more technical) curvature bounds in \cite{Z2,RZ11,BZ}.
It is not needed for the cases $k=d$, cf.~\cite{gatz}, and $k=d-1$, see \cite{rw09}. A rather general assumption for $k\leq d-2$ -- used in \cite{RZ11} -- is the following integrability condition: There are constants $a,\eps_0>0$ such that
\begin{align} \label{eq:ic}
  \sup_{\delta\leq\eps_0}\frac 1{|\ln\delta|}\int_\delta^{\eps_0} \eps^{-k}\sup_{x\in F} C_k^\var(F_\eps,B(x,a\eps))\frac{d\eps}{\eps}<\infty\,.
\end{align}
This assumption does only ensure existence of average limits, but not of the limits in \eqref{eq:fc}, for which slightly stronger assumptions are required, e.g.\ the curvature bounds used in \cite{Z2} and \cite{WZ}, which are equivalent to the following condition, see \cite[Remark 3.1.3]{RZ11}: There are constants $a>1$ and $\eps_0>0$ such that
 \begin{align} \label{eq:cbc}
  \esssup{\eps\in(0,\eps_0], y\in F} \eps^{-k} C_k^\var(F_\eps,B(y,a\eps))<\infty\,.
\end{align}
The existence of associated fractal curvature measures is discussed in \cite{W1,WZ,RZ11} and corresponding results for curvature-direction measures are obtained in \cite{BZ}, where an integrability assumption (on the curvature direction measures) is used that is even weaker than \eqref{eq:ic}. Generalizations to self-conformal sets are studied in \cite{Kom,KK,B}.

As we have just outlined, previous work on fractal curvatures has concentrated
on the case $s=D$, leaving to the side the fact that $D$ is not always the right choice for the scaling exponents.
Indeed, a simple class of examples are fulldimensional cubes $Q$ in $\R^d$, cf.~\cite[Ex.~2.3.5]{W1}, which can be generated as self-similar sets and for which the choice $s=k$ is optimal for the $k$-th fractal curvature for all $k$, in that $\sC_k^k(Q)$ is positive and finite, while $\dim_M Q=d$ for these cubes. Some further, less trivial examples of self-similar sets for which some scaling exponent is different from the dimension are presented in Section~\ref{sect:2}. They motivate the investigations in this paper. It is one of our main objectives, to understand when such a nongeneric behaviour occurs for self-similar sets.

For this purpose, we need a satisfactory definition of the  $k$-th scaling exponent of a set $F\subset\R^d$, which is not easy to give in general. Roughly, it is the number $t$ for which $\sC^t_k(F)$ is nonzero and finite. There is not always such a $t$, but if there is, then $\sC_k^s(F)=0$ for all $s>t$ and $\sC_k^s(F)=\pm\infty$ for all $s<t$. Additional difficulties are that the curvature measures $C_k(F_\eps,\cdot)$ are signed and that therefore $C_k(F_\eps)$ may change its sign infinitely many times as $\eps$ tends to $0$ or even vanish for all $\eps>0$, and that the measures may not be defined for each $\eps>0$. This is partially resolved by working with the total variation measure $C_k^\var(F_\eps,\cdot)$ of $C_k(F_\eps,\cdot)$. Denote by $C_k^\var(F_\eps)$ its total mass. Two possible definitions arise more or less naturally from the limits in \eqref{eq:fc} and \eqref{eq:afc}. As will become clear later, each has its advantages and disadvantages and it is not clear which one is the best notion of scaling exponent. Let $F\subset\R^d$ be a bounded set for which the $k$-th curvature measure of $F_\eps$ is defined for almost all $\eps>0$. The first possibility is to define the \emph{$k$-th scaling exponent} of $F$ by
\begin{align} \label{eq:sk-def1}
  s_k(F)&:=\inf\{s\ge 0: \esslim{\eps\to 0} \eps^{s-k} C_k^{\var}(F_\eps) =0\}\\
&\phantom{:}=\sup\{s\ge 0: \esslimsup{\eps\to 0} \eps^{s-k} C_k^\var(F_\eps)=\infty\}\notag
\end{align}
which generalizes the definition suggested in \cite{W1,W3}. The second possibility is to use again some averaging and consider the number
\begin{align} \label{eq:sk-def2}
  a_k(F)&:=\inf\{s\ge 0: \lim_{\delta\to 0} \frac 1{|\ln\delta|}\int_\delta^1\eps^{s-k-1} C_k^\var(F_\eps) d\eps=0\}\\
&\phantom{:}=\sup\{s\ge 0: \limsup_{\delta\to 0} \frac 1{|\ln\delta|}\int_\delta^1\eps^{s-k-1} C_k^\var(F_\eps) d\eps=\infty\}.\notag
\end{align}
as the $k$-th scaling exponent. We will use the term \emph{average scaling exponent} for $a_k(F)$ in order to be able to distinguish both exponents. Regarding their relation, we point out that in general one has the inequality $a_k(F)\leq s_k(F)$, but both exponents need not coincide, as the example of the Cantor dust $C\times C\subset\R^2$ (where $C\subset\R$ is the middle third Cantor set) illustrates, for which $s_0(C\times C)=\infty$ while $a_0(C\times C)$ equals the dimension of this set, see~\cite[Example 4.2]{RZ11}.

Note that both exponents $s_k(F)$ and $a_k(F)$ are upper exponents. Corresponding lower exponents can be defined by replacing the (upper) limits in the definitions by lower limits. However, at present we see no application of these lower exponents. Below our results will be formulated for the exponents $s_k(F)$, although most of them hold equally for the exponents $a_k(F)$. This is due to the fact that most results are derived from estimates for the curvature measures which hold for (almost) all $\eps>0$ allowing to draw conclusions for both exponents. Only in cases, where the curvature conditions \eqref{eq:ic} or \eqref{eq:cbc} are involved, the results for $a_k(K)$ and $s_k(K)$ may differ. For instance, for self-similar sets $K$ the integrability condition \eqref{eq:ic} ensures the inequality $a_k(K)\leq D$, while only the stronger condition \eqref{eq:cbc} ensures $s_k(K)\leq D$. Condition \eqref{eq:ic} is not sufficient for this conclusion as the above example of the Cantor dust demonstrates. By definition, we also have $0\leq a_k(F)\leq s_k(F)$ in general.

We point out that, for $k=d$ and $k=d-1$, the essential limits in \eqref{eq:sk-def1} are ordinary limits and thus, for $k=d$, we recover the upper Minkowski dimension: $s_d(F)=\udim_M F$. The exponent $s_{d-1}(F)=\udim_S F$ is also known as the \emph{upper S-dimension} of $F$, cf.~\cite{rw09, rw11}. Since for self-similar sets $K$ satisfying OSC, Minkowski and S-dimension are known to exist, one even has $s_d(K)=\dim_M (K)$ and $s_{d-1}(K)=\dim_S K$ in this case. Moreover, it is not difficult to see that $a_k(K)=s_k(K)$ for $k=d$ and $k=d-1$.

In this paper we study the situation when some of the scaling exponents do not coincide with the dimension.
In Section~\ref{sec:loc-flat}, we introduce the notion of \emph{local flatness} (see Definition~\ref{def:loc-flat}) in order to characterize such nongeneric behaviour of the scaling exponents. We conjecture that a self-similar set is locally flat if and only if some of its scaling exponents differ from its dimension, see~Conjecture~\ref{conj}. As a main result of this paper, we resolve this conjecture for self-similar sets in $\R$ and $\R^2$, see Corollary~\ref{cor:main-R} and Theorem~\ref{thm:main-R2}, respectively. Corollary~\ref{cor:main-R} is essentially a special case of Theorem~\ref{fulldimension}, which resolves the conjecture for fulldimensional self-similar sets in $\R^d$.
One of the questions that arise on the way (and which is of independent interest) is, whether scaling exponents are independent of the dimension of the ambient space (in which the parallel sets are taken). Proposition~\ref{subsets_of_R} shows the independence in the case required for the derivation of the main results.

In Section~\ref{sec:s0}, we study the $0$-th scaling exponent for sets in $\R^d$ and obtain some sufficient conditions for this exponent to be equal to the dimension of the set. In particular, the disconnectedness of the complement or the total disconnectedness of the set itself are sufficient, cf.~Theorem~\ref{thm:conn-compl} and Corollary~\ref{cor:disconn}, respectively. More precisely, we obtain these results for the directional variant $\tilde s_0(F)$ of $s_0(F)$, which is introduced as follows. Let $\widetilde C_k(F_\eps,\cdot), k=0,1,\ldots,d-1$, be the $k$-th curvature-direction measure of $F_\eps$ (on the normal bundle $\nor(F_\eps)\subset\R^d\times S^{d-1}$ of $F_\eps$); for completeness let $\widetilde C_d(F_\eps,\cdot):=\lambda_d(F_\eps\cap \pi_1(\cdot))$, where $\pi_1$ is the projection onto the space component. Recall that $C_k(F_\eps, A)=\widetilde C_k(F_\eps, A\times S^{d-1})$, for any Borel set $A\subset\R^d$. Let $\widetilde C_k^\var(F_\eps,\cdot)$ denote the total variation measure and $\widetilde C_k^\var(F_\eps)$ its total mass. Then $\tilde s_k(F)$ and $\tilde a_k(F)$ are introduced by replacing in \eqref{eq:sk-def1} and \eqref{eq:sk-def2} $C_k^\var(F_\eps)$ by $\widetilde C_k^\var(F_\eps)$, i.e.,
\begin{align} \label{eq:tilde-sk-def1}
\tilde s_k(F):=\inf\{s\ge 0: \esslim{\eps\to 0} \eps^{s-k} \widetilde C_k^\var(F_\eps) =0\}.
\end{align}
From the relation $C_k^\var(F_\eps)\leq \tilde C_k^\var(F_\eps)$ it is easily seen that
the inequalities $s_k(F)\leq \tilde s_k(F)$ and $a_k(F)\leq \tilde a_k(F)$ are true, whenever one of these scaling exponents (and thus the others) are defined. It is believed that one has in fact equality in these relations in general, which is up to now only clear for the cases $k=d$ and $k=d-1$. For sets in $\R^2$, we show equality of the scaling exponents in Corollary~\ref{exp=direxp} below. The main reason for switching to the directional exponents is that one can take advantage of the integral representations derived by Zähle~\cite{Z1} for curvature-direction measures, see e.g.\ Lemma~\ref{lem:C0var}.

Section~\ref{sec:plane} is devoted to the resolution of Conjecture~\ref{conj} in dimension 2, for which many of the results of Sections~\ref{sec:loc-flat} and \ref{sec:s0} are employed.
Originally we have started this investigation of the  variability of the scaling exponents from a slightly different point of view, which can be summarized in the following question: Given a vector $(t_0,\ldots,t_d)\in\R^{d+1}$, does there exist a self-similar set $K\subset\R^d$ such that $s_k(K)=t_k$ for $k=0,\ldots,d$? That is, is it possible to prescribe the scaling exponents and construct a set with exactly those exponents?
Our results indicate that the family of vectors, for which such self-similar sets exist is very sparse. We have added some discussion of this in Section~\ref{sec:final}. However, we are still far from a complete answer. The same question may be asked for arbitrary sets in $\R^d$. We demonstrate in Example~\ref{ex:general} that there is more freedom for the choice of the scaling exponents when the self-similarity assumption is dropped.
Furthermore, as a byproduct of the proof of Theorem~\ref{thm:conn-compl}, we obtain in Theorem~\ref{thm:tiling} that a self-similar sets possesses a compatible self-similar tiling (in the sense of \cite{PW}) if and only if its complement is disconnected, resolving thus an open question in \cite{PW}.

\section{Preliminaries} \label{sec:pre}

\paragraph{\bf Curvature measures of parallel sets.} Denote the {\it closure of the complement} of a compact set $F$ by $\widetilde{F}$.
A distance $\eps\ge 0$ is called {\it regular} for the set $F$ if $\widetilde{(F_\eps)}$ has positive reach in the sense of Federer \cite{F59} and the boundary $\partial F_\eps$ is a Lipschitz manifold. A sufficient condition for regularity in this sense is that $\eps$ is a regular value of the distance function of $F$ in the sense of Morse theory, cf.~Fu \cite{Fu}. It is well known from the latter paper that, for $F\subset\R^d$ with $d\leq 3$, this property satisfied for Lebesgue almost all $\eps$. For regular $\eps$, the {\it curvature measures} of the sets $\widetilde{(F_\eps)}$ (which have positive reach) are well defined in the sense of Federer \cite{F59} and therefore the curvature measures of $F_\eps$ are determined via normal reflection:
\begin{equation}
C_k(F_\eps,\mydot):=(-1)^{d-1-k}C_k(\widetilde{F_\eps},\mydot)\, ,~~k=0,\ldots,d-1\,.
\end{equation}
For $k=d-1$, the surface area is included which coincides for $F_\eps$ and $\widetilde{(F_\eps)}$. For completeness, the volume measure $C_d(F_\eps,\mydot):=\lambda_d(F_\eps\cap\mydot)$ is added to this list. Let $C_k(F_\eps):=C_k(F_\eps,\R^d)$ be the total mass of $C_k(F_\eps,\cdot)$.

We recall some of the basic properties of curvature measures. Let $X,Y\subset\R^d$ be sets with positive reach. Then
\begin{enumerate}
\item $C_k(X,A)=C_k(g(X),g(A))$ for every Euclidean motion $g$ (\emph{motion covariance}),
\item $C_k(X\cup Y,\cdot)=C_k(X,\cdot)+C_k(Y,\cdot)-C_k(X\cap Y,\cdot)$, provided $X\cup Y$ has positive reach (implying that also $X\cap Y$ has positive reach, see \cite[Theorem 5.16]{F59}) (\emph{additivity}),
\item $C_k(\lambda X,\lambda \cdot)=\lambda^{k} C_k(X,\cdot)$ for every $\lambda>0$ (\emph{$k$-homogeneity}),
\item If $X\cap U = Y\cap U$ for some open set $U$, then $C_k(X,\cdot)|_{U}=C_k(Y,\cdot)|_U$ (\emph{local determinacy}),
\item $C_0(X)=\chi(X),$ provided $X$ is compact and $\chi$ denotes the Euler characteristic (\emph{Gauss-Bonnet theorem}).
\end{enumerate}

Beside curvature measures $C_k(X,.)$, we will also consider their directional variants $\gC_k(X,.).$
The \emph{curvature-direction measures} (or \emph{generalized curvature measures}) $\gC_k(X,.)$, $k=1,\ldots,d-1$ do not live on the boundary of $X$, but on the normal bundle of $X$, defined by
$$
\nor X:=\{(x,n)\in \partial X\times S^{d-1}: n\in\Nor(X,x)\},
$$
where $\Nor(X,x):=\{n\in \R^d:\langle n,y\rangle\leq 0 \text{ for all } y\in \Tan(X,x)\}$ is the \emph{normal cone} and  $\Tan(X,x)$ the \emph{tangent cone} of $X$ at a point $x\in X$, see e.g.~\cite[§4.3]{F59} for more details. Let $\pi_{1}:\er^{d}\times S^{d-1}\to \er^d,(x,n)\mapsto x$ and $\pi_{2}:\er^{d}\times S^{d-1}\to S^{d-1},(x,n)\mapsto n$ be the projections onto the first and the second component, respectively.
If $X\subset\R^d$ is a full-dimensional set, then $C_k(X,.)$ can be interpreted as the projection of the measure $\gC_k(X,.)$ with respect to $\pi_1$, that is $C_k(X,\cdot)=\gC(X,\cdot\times S^{d-1})=$
One of the advantages of the measures $\gC_k$ is the validity of the following integral formula due to Z\" ahle (see \cite[Theorem 3]{Z1}): If $X\subset\R^d$ has positive reach and $k\in\{0,\ldots,d-1\}$, then for any Borel set $B\subset\R^d\times S^{d-1}$
  \begin{align*}
     \gC_k(X,B)=c_k^{-1} \int_{\nor X} \ind{B}(x,n) \frac{\displaystyle \sum_{1\leq i_1<\dots<i_{d-1-k}\leq d-1}\prod_{j=1}^{d-1-k}\kappa_{i_j}(x,n)}{\prod_{j=1}^{d-1} \sqrt{1+\kappa_j^2(x,n)}} \Ha^{d-1}(d(x,n)).
  \end{align*}
Here, $\kappa_i(x,n)$ are the (generalized) principal curvatures corresponding to $(x,n)\in \nor X$, $\ind{B}$ is the characteristic function of $B$, and $c_k=(d-k)\alpha(d-k)$, with $\alpha(j)$ being the volume of the $j$-dimensional unit ball. 
Curvature-direction measures $\gC_k$ have similar properties as the ones listed above for $C_k$.
For more details and background on curvature measures, we refer to \cite{Z1,Z2} and the references given therein.\\

\medskip

\paragraph{\bf Self-similar sets.} The main object of study are self-similar sets satisfying the open set condition which we recall now briefly, introducing also some notation this way, which will be used throughout.

Suppose that $N\in\N$, $N\geq 2$. For $i=1,\ldots,N$, let $S_i:\R^d\to\R^d$ be a contracting similarity with contraction ratio $0<r_i<1$.
Then there is a unique nonempty compact set $K\subset\R^d$ invariant under the set mapping ${\bf S} (\mydot):=\bigcup_i S_i(\mydot)$. This set is known as the \emph{self-similar set} generated by the function system (shortly IFS) $\{S_1,\ldots, S_N\}$, cf.~\cite{H}.
The set $K$ (or, more precisely, the system $\{S_1,\ldots,S_N\}$) is said to satisfy the \emph{open set
condition} (OSC) if there exists a nonempty open set $O\subset\R^d$ such that
 $$
 \bigcup _i S_i O \subseteq O \quad \text{ and } \quad S_i O \cap
S_j O=\emptyset  \text{ for } i\neq j\,.
$$
Such a set $O$ is sometimes called a \emph{feasible open set} of the IFS $\{S_1,\ldots,S_N\}$, or of $K$.
The \emph{strong open set condition} (SOSC) holds for $K$ (or $\{S_1,\ldots,S_N\}$), if there exists a feasible open set $O$ which additionally satisfies $O\cap K\neq\emptyset$.
It was shown by Schief \cite[Theorem~2.2]{S}, that in $\R^d$ OSC and SOSC are equivalent, i.e., for $K$ satisfying OSC, there exists always a feasible open set $O$ with $O\cap K\neq\emptyset$.

The unique solution $s=D$ of the equation $\sum_{i=1}^N r_i^s=1$ is called the \emph{similarity dimension} of $K$.  It is well known that for any self-similar set $K$ satisfying OSC, $D$ coincides with both the Minkowski dimension $\dim_M K$ and the Hausdorff dimension $\dim_{H}K$ of $K$, which in particular means that in this case $\dim_M K=\dim_{H}K$ holds, cf.\ e.g.~\cite[Theorem 9.2]{fal}.

Let $\Sigma^m_N:=\{1,\ldots,N\}^m$ be set of all words of length $m$ over the alphabet $\{1,\ldots,N\}$ and denote $\Sigma^*_N:=\bigcup_{j=0}^\infty\Sigma^j_N$.
For $\omega=\omega_1\ldots \omega_n\in\Sigma^*_N$ we denote by $|\omega|$ the \emph{length of $\omega$}
(i.e., $|\omega|=n$) and by $\omega|k:=\omega_1\ldots \omega_k$ the subword of the first $k\le n$ letters.
We often abbreviate $r_\omega:=r_{\omega_1}\ldots r_{\omega_n}$ or $S_\omega:=S_{\omega_1}\circ\ldots\circ S_{\omega_n}$ and similarly for other notions concerning self similar sets.
Furthermore, let $r_{\min}:=\min\{r_i: i=1,\ldots, N\}$.

\medskip

\paragraph{\bf Further notation.} Throughout we use the following more or less standard notation without further mention.
For $x\in\er^d$ and $\eps>0$ we denote $B(x,\eps)$ the closed ball with centre $x$ and radius $\eps$. For the topological boundary of a set $A\subset\er^d$ we write $\partial A$. $A^\circ$ and $A^c$ are used for the interior and the complement of $A$, respectively. The $s$-dimensional Hausdorff measure is denoted by $\H^s$ and $S^{d-1}$ is the unit sphere in $\er^d.$

\section{Basic examples} \label{sect:2}

We start with a simple example of a class of self-similar sets for which the $0$-th scaling exponents are not equal to the dimension.

\begin{figure}
\begin{minipage}{60mm}
  \includegraphics[width=50mm]{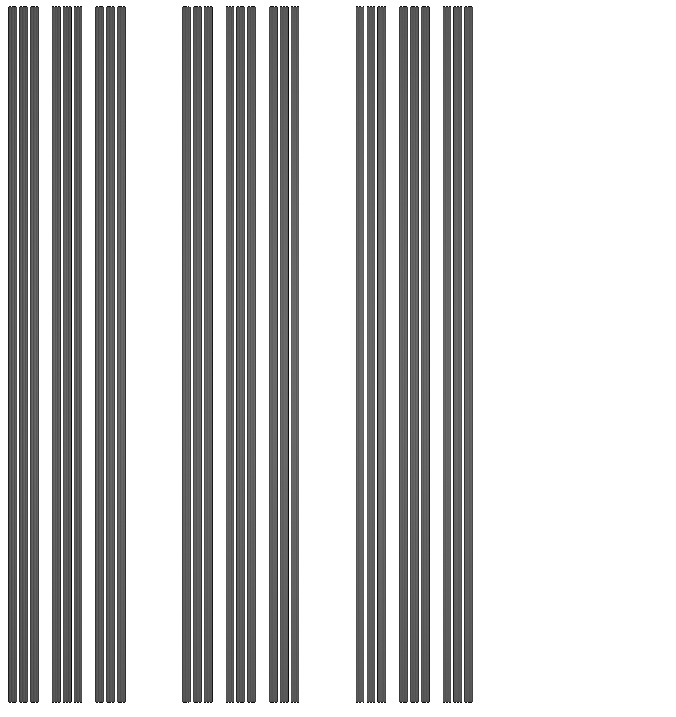}
 \end{minipage}
\begin{minipage}{65mm}
  \includegraphics[width=70mm]{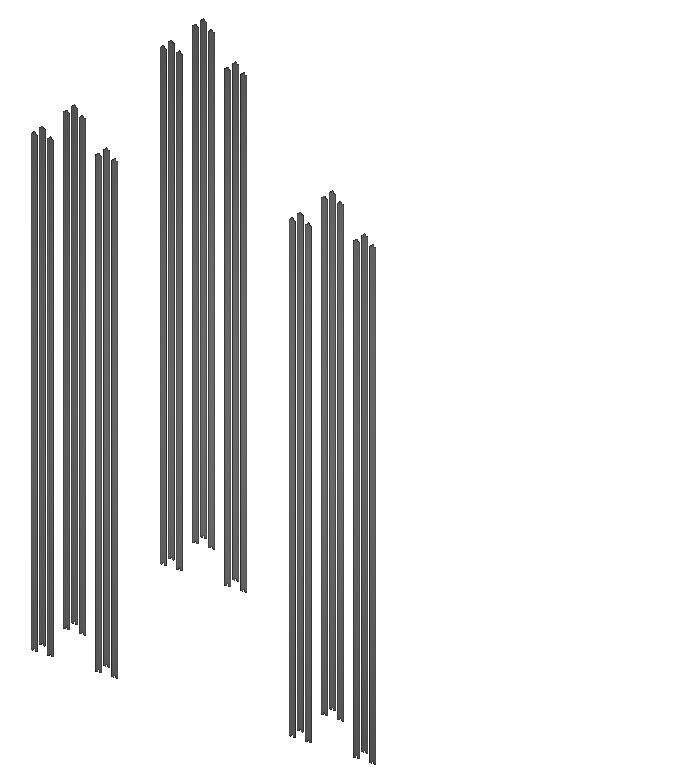}
\end{minipage}
\caption{\label{fig:ex1}(Left): The set $F^{4,3}=K^{4,3}\times[0,1]$ from Example~\ref{ex:1}. (Right): The set $\hat F^{4,3}=\hat K^{4,3}+(\{0\}\times[0,1])$ from Example~\ref{ex:locallyflat_but_not_globally} with $t_1=0$, $t_2=1/6$ and $t_3=-1/6$.}
\end{figure}

\begin{ex} \label{ex:1}
There is a set $N\subset [0,1]$, dense in $[0,1]$, such that for each $a\in N$ there exists a self-similar set $F=F(a)\subset\er^{2}$ such that $s_{0}(F)=a$ and $s_{1}(F)=s_{2}(F)=a+1.$
\end{ex}

\begin{proof}
Let $n,m\in \N$ with $m<n.$ For $k,l=1,...,n$, define $\varphi_{k,l}:\R^2 \to \R^2$ by
\begin{equation*}
\varphi_{k,l}(x,y)=\frac{1}{n}\bigl(x+k-1,y+l-1\bigr).
\end{equation*}
Let $K^{n,m}$ and $F^{n,m}$ be the self-similar sets generated by the mappings $\varphi_{k,l},$ $k=1,...,m$, $l=1$ and $\varphi_{k,l},$ $k=1,...,m$, $l=1,...,n$, respectively; see~Figure~\ref{fig:ex1} (left) for an illustration. (If $K^{n,m}$ is viewed as a subset of $\R$, then we have $F^{n,m}=K^{n,m}\times[0,1]$, however, in the sequel, $K^{n,m}$ is studied as a subset embedded in $\R^2$.)
Since both $K^{n,m}$ and $F^{n,m}$ satisfy OSC, we have $\dim_M K^{n,m}=\frac{\log m}{\log n}$ and $\dim_M F^{n,m}=1+\frac{\log m}{\log n}$.
Let
\begin{equation*}
N:=\left\{\frac{\log m}{\log n}:n,m\in \N , m<n\right\}.
\end{equation*}
Using the inequality
\begin{equation*}
\left|\frac{\log m+1}{\log n}-\frac{\log m}{\log n}\right|=\frac{\log \frac{m+1}{m}}{\log n}\leq\frac{2}{\log n}
\end{equation*}
and observing that
$\frac{m-1}{m}\nearrow 1$ and $\log n\to\infty$ as $n\to\infty$,
it is easily seen that $N$ is dense in $[0,1].$

Choose $a=\frac{\log m}{\log n}\in N$ and consider the set $F=F(a):=F^{n,m}.$
First observe that $C_{0}^\var((K^{n,m})_{\varepsilon})=C_{0}^\var((F^{n,m})_{\varepsilon})$ for every $\varepsilon>0$ and therefore $s_{0}(K^{n,m})=s_{0}(F^{n,m})$.
Let $U:=(-\frac{1}{n},1-\frac{1}{n})\times\er$. Then $U$ is a feasible open set for $K^{n,m}$. 
Put $\varepsilon_{0}=\frac{1}{3n}$ and $B=B(0,\varepsilon_{0})\subset U_{-\varepsilon_{0}}.$
Then, for every $r\in(0,\varepsilon_{0}]$ we have the arc $S(r):=\bd B(0,r)\cap \left(-\infty,0\right]\times\left(-\infty,0\right]$ contained in $B\cap \bd (K^{n,m})_{r}.$
Note that $C_{0}((K^{n,m})_{r},S_{r})=\frac{1}{4}$.
Now it is sufficient to use Proposition~\ref{prop:w1-2-3-8} (or Theorem~2.3.8 in \cite{W1}) for $B$, $\varepsilon_{0}$ as above and $\beta=\frac{1}{4}$.
We obtain
\begin{equation*}
s_{0}(F^{n,m})=s_{0}(K^{n,m})=\dim_M(K^{n,m})=\frac{\log m}{\log n}.
\end{equation*}
Finally, since $F^{n,m}$ has empty interior, we use \cite[Corollary~3.4]{rw09} to obtain $s_{1}(F^{n,m})=s_{2}(F^{n,m})=\dim_M(F^{n,m}).$

Summing up, we proved that for $a=\frac{\log m}{\log n}\in N$ the vector of scaling exponents of the set $F(a)$ is equal to
$
(a,1+a,1+a).
$
\end{proof}

The next class of examples in the plane deals with the difference between the Minkowski dimension of a self-similar set and the Minkowski dimension of its boundary (resulting in a difference between $s_{2}(K)$ and $s_{1}(K)$).
Obviously, such a difference can only occur if the set has interior points, that is, if the set has the dimension of the ambient space.

\begin{figure}
\begin{minipage}{62mm}
  \includegraphics[width=62mm]{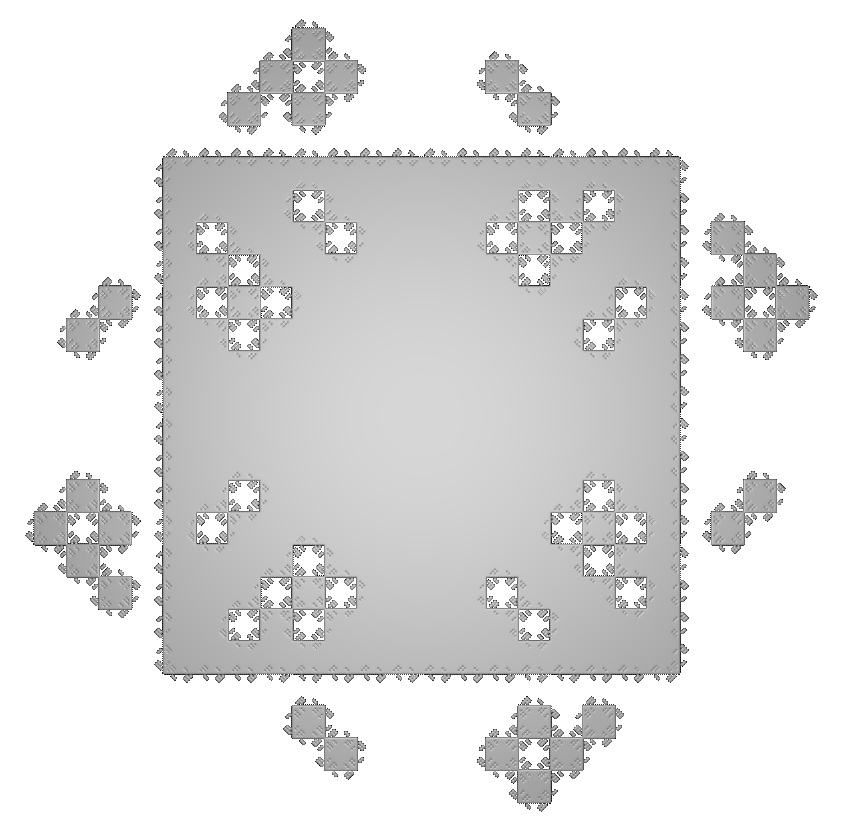}
 \end{minipage}
\begin{minipage}{62mm}
  \includegraphics[width=62mm]{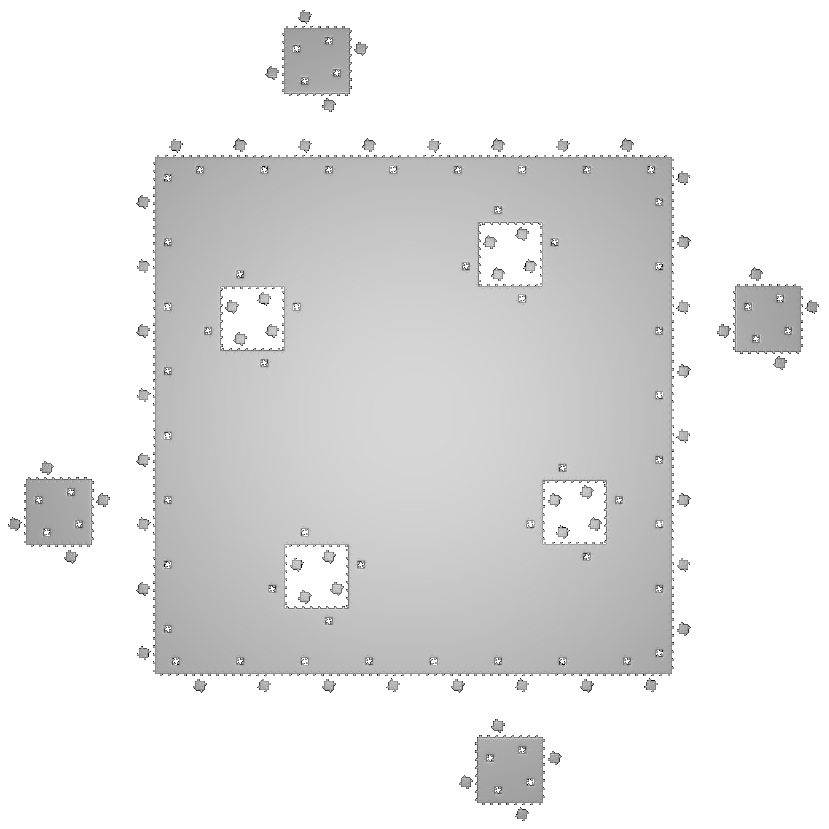}
\end{minipage}
\caption{\label{fig:ex2}Two realizations of the self-similar sets $K^{k,m}$ discussed in Example~\ref{ex:2} for $k=4$, $m=7$ (left) and for  $k=3$, $m=1$ (right).}
\end{figure}

\begin{ex}  \label{ex:2}
There is a dense subset $M$ of $[1,2]$ such that for each $a\in M$ there exists a self-similar set $F=F(a)\subset\R^2$ with $\dim_M F=2$ and $\dim_M \partial F=a$.
\end{ex}

\begin{proof}
The basic idea is to subdivide the square $W:=[0,1]^2$ into $2^{2k}$ sub-squares each of side length $2^{-k}$. Some of them are kept in their position and the others are rotated to the outside of $W$ in such a way that the OSC is not violated; see~Figure~\ref{fig:ex2} for an illustration.
This procedure can be described by an IFS consisting of $4^k$ similarities each with ratio $2^k$ and such that each square is the image of $W$ under one of these mappings. While the dimension of the generated self-similar set is always $2$, the number of rotated squares will determine the dimension of its boundary.

Let
$f_{1}(x,y)=\frac{1}{2}(x-\frac{1}{2},y-\frac{1}{2})$, $f_{2}(x,y)=\frac{1}{2}(x+\frac{1}{2},y-\frac{1}{2})$,
$f_{3}(x,y)=\frac{1}{2}(x-\frac{1}{2},y+\frac{1}{2})$, $f_{4}(x,y)=\frac{1}{2}(x+\frac{1}{2},y+\frac{1}{2})$
and let $g_{i},$ $i=0,\dots3$ be the rotation around the point $S^{i}(\frac 12,0)$ by $\pi,$
where $S$ is the rotation around the point $(0,0)$ by $\frac{\pi}{2}.$

For $k\geq 2$, we define an equivalence relation $\approx_{k}$ on the set $\{1,\dots,4\}^{k}$ as follows:
\begin{equation*}
  \omega\approx_{k} \sigma \text{ if and only if } f_{\omega}(W)=S^{i}\circ f_{\sigma}(W) \text{ for some } i\in\en
\end{equation*}
Note that each equivalence class of $\approx_{k}$ consists of 4 words each corresponding to a square (of side length $2^{-k}$) in a different image $f_i(W)$ of $W, i=1,\ldots, 4$.

Let $\Q'_{k}$ denote the system of equivalence classes of $\approx_{k}$. For a word $\omega\in\{1,\ldots, 4\}$, we write $Q(\omega)$ for the class in $\Q'_k$ containing $\omega$.
Let the subsystem $\Q''_{k}$ of $Q'_{k}$ be defined by
\begin{equation}
  \Q''_{k}:=\{Q\in \Q'_{k}| f_{\omega}(W)\cap \bigcup_{i=1}^4 \bd f_i(W) =\emptyset \text{ for some } \omega\in Q \}.
\end{equation}
That is, in $\Q''_{k}$ we exclude all squares touching the boundary of one of the sets $f_i(W)$. For reasons of symmetry this is consistent with the equivalence relation in the sense that $Q\in\Q''_k$ implies that for each  $\omega\in Q$, the square $f_\omega(W)$ does not intersect the set $\bigcup_{i=1}^4 \bd f_i(W)$.
We will also need to avoid diagonals and so
we consider the system $\Q'''$ which is equal to $\Q''$ without the squares that have the centre on both diagonals $\{x=y\}$ and $\{x=-y\}$.
The cardinality of $\Q'''_{k}$ is $|\Q'''_{k}|=4^{k-1}-2^{k+1}+2.$
Finally, we place a chessboard pattern over the remaining squares and exclude all the black squares. This can be done in a $\approx_k$-consistent way as follows: Let $\Q_{k}$ the subsystem of all $Q\in\Q'''_{k}$ such that for each $\omega=\omega_1\omega_1 \ldots \omega_k\in Q$ the number $\omega_1+\omega_k$ is even.
Note that $|\Q_{k}|=\frac 12 |\Q_{k}'''|=2\cdot4^{k-2}-2^{k}+1$. Moreover, $\Q_k$ has the following \emph{`chessboard'} property: For any $P,Q\in \Q_{k}$, $\omega\in P$, $\sigma\in Q$ the white squares $f_{\omega}(W)$ and $f_{\sigma}(W)$ have no common side.

Fix some integers $k\geq 2$ and $0\leq m\leq |\Q_{k}|$. Choose an arbitrary subsystem $\Q^{m}_{k}$ of $\Q_{k}$ with $|\Q^{m}_{k}|=m.$
This can be interpreted as a coloring the corresponding amount of white chess-tiles with a third color.
First define
$$
H:=\{(x,y)\in(0,\infty)\times\er: x\leq|y|\}
$$
and put $H_i:=S^{i}(H)$, $i=0,...,3.$
Then define a mapping $\Psi:\{1,...,4\}^{k}\to\{0,...,3\}$ such that $\Psi(\omega)=i$ if and only in $f_{\omega}((0,0))\in H_i.$
Define a mapping $h^{k,m}_{\omega}:\R^2\to\R^2$ for $\omega\in\{1,\dots,4\}^{k}$ by
\begin{equation*}
h^{k,m}_{\omega} = \left\{
  \begin{array}{l l}
    g_{\Psi(\omega)}\circ f_{\omega}, &\text{if } Q(\omega)\in\Q^{m}_{k},\\

    f_{\omega}, & \text{otherwise}.\\
  \end{array} \right.
\end{equation*}
Let $K^{k,m}$ be the self-similar set generated by the IFS $\{h^{k,m}_{\omega}\mid \omega\in\{1,\dots,4\}^{k}\}$.
First, it is easy to see that $K^{k,m}$ satisfies the open set condition (using e.g.\ the finite clustering property in \cite[Theorem~2.2(v)]{S}),
and, since $K^{k,m}$ is generated by $4^{k}$ mappings each with similarity ratio $2^{-k}$, we obtain $\dim_M K^{k,m}=2.$

Moreover, one can see that $\partial K^{k,m}$ is a union of four mutually isometric self-similar sets (each of them naturally corresponding to one side of $W$).
Each of them, say $B_1$ is a similar copy on a self similar set generated by the similarities
\begin{align*}
F_{l}(x,y)&=(2^{-k}x+l-\frac{1}{2},2^{-k}y),&l=0,\ldots,2^{k}-1,\\
G_{\omega}^{i}&=T_{\omega}\circ 2^{-k}S^{i}\circ T_{(0,-\frac{1}{2})}, & Q(\omega)\in \Q^{m}_{k}, \Psi(\omega)=3, i=0,\dots,3,\\
H_{\omega}^{i}&=g_3\circ T_{\omega}\circ 2^{-k} S^{i}\circ T_{(0,-\frac{1}{2})}, & Q(\omega)\in \Q^{m}_{k}, \Psi(\omega)=3, i=0,\dots,3,
\end{align*}
where $T_{(0,-\frac{1}{2})}$ is a translation by the vector $(0,-\frac{1}{2})$ and $T_{\omega}$ is a translation by vector $h^{k,m}_{\omega}((0,0))-(0,-\frac{1}{2}).$
In this IFS there are $2^{k}$ mappings $F_{l}$, $4m$ mappings $G_{\omega}^{i}$ and $4m$ mappings $H_{\omega}^{i}$ all of them with similarity ratio $2^{-k}.$
Note that the IFS satisfies OSC with $O:=\co\{\pm(\frac{1}{2},0),\pm(0,\frac{1}{2})\}^{\circ}$ being a feasible open set, which implies that
\begin{equation*}
\dim_M \partial K^{k,m}=\frac{\log(2^{k}+8m)}{\log(2^{k})}.
\end{equation*}
Finally observe that, similarly as in the previous example, the set
\begin{equation*}
M:=\left\{\frac{\log(2^{k}+8m)}{\log(2^{k})}:k\geq2, 0\leq m\leq 2\cdot4^{k-2}-2^{k}+1\right\}
\end{equation*}
is dense in the interval $[1,2].$
\end{proof}

Note that the possible values of scaling exponents in the examples above are fairly restricted by the OSC and it is not clear whether there exist self-similar sets (satisfying OSC) for each value in the intervals $[0,1]$ or $[1,2]$ instead of just the dense sets $N$ and $M$, respectively.

\section{Local flatness} \label{sec:loc-flat}

The above examples motivate the following definition and conjecture.

\begin{definition} \label{def:loc-flat}
Let $K\subset\er^{d}$ and $m\in\{1,\ldots, d\}$. We say that $K$ is \emph{locally $m$-flat} if for every $x\in K$ and $\eps>0$ there is a closed cube $C\subset B(x,\eps)$ such that $C^\circ\cap K$ is nonempty (where $C^\circ$ denotes the interior of $C$) and $C\cap K$ is similar to $[0,1]^{m}\times P$ for some set $P\subset\er^{d-m}$. (For $m=d$ this should be interpreted as $C\cap K$ being similar to $[0,1]^d$.)
We say that $K$ is \emph{locally flat}, if it is locally $m$-flat for some $1\leq m\leq d$.

In contrast to this local notion, we call the set $K$ \emph{(globally) $m$-flat} if there exists a set $Q\subset\er^{d-m}$ such that $K$ itself is similar to $[0,1]^m\times Q$. That is, a set is globally flat if it is a product set with one factor being a cube (of some dimension).
\end{definition}

Note that all sets in Example~\ref{ex:1} are locally $1$-flat, while all sets in Example~\ref{ex:2} are locally $2$-flat.
It is clear that global $k$-flatness implies local $k$-flatness, but the converse is not necessarily true (see Example~\ref{ex:locallyflat_but_not_globally} below).
Note that we do not require the cube $C$ in the definition of local flatness to contain the point $x$. Moreover, for any $m\geq 1$, local $m+1$-flatness implies local $m$-flatness. We consider this notion important for self-similar sets, because we conjecture local flatness to be characteristic for the occurrence of scaling exponents that are strictly smaller than the dimension. More precisely, we believe the following is true:

\begin{conjecture}\label{conj}
Let $K$ be a self-similar set in $\er^{d}$ satisfying OSC. Then $s_{k}(K)\geq\dim_M K$ for every $k\in\{0,1,\ldots,d\}$ if and only if $K$ is not locally flat.
If $K$ satisfies additionally the integrability conditions \eqref{eq:ic}, then the assertion holds with `$=$'-signs instead of `$\geq$'
\end{conjecture}

In the sequel we will confirm this conjecture for sets in $\R$ and $\R^2$ and we will give some support of this conjecture in higher dimensions.

Our first goal is to get a better understanding of the notion of local flatness in connection with self-similarity.
The first statement shows that to decide the  local flatness of a self-similar set it is enough to look at a single point $x$ of the set (and at a fixed ball $B(x,\eps)$). Local flatness `at one point' (and at a fixed scale) implies local flatness everywhere in the set (and at all scales).
\begin{prop}\label{prop:one-point-flat}
Assuming OSC the following assertions are equivalent for a self-similar set $K\subset\R^d:$
\begin{enumerate}
\item[(i)] $K$ is locally $k$-flat.
\item[(ii)] There are $x\in K$, $\eps>0$ and a cube $C\subset B(x,\eps)$
such that $C^\circ\cap K$ is nonempty and $C\cap K$ is similar to $[0,1]^{k}\times P$ for some $P\subset\er^{d-k}$.
\end{enumerate}
\end{prop}

\begin{proof}
The implication $(i)\Rightarrow (ii)$ is trivial.
To prove the reverse implication, let $\{\varphi_1,\ldots,\varphi_N\}$ be similarities generating $K$, let $O$ be a strong feasible open set for $K$ (i.e., a feasible open set such that $K\cap O \neq\emptyset$) and let $z$ be a point in $K\cap O$. Since the iterates $\varphi_\sigma(z)$, $\sigma\in\Sigma_N^*$ of $z$ are dense in $K$ and since they are all contained in $K\cap O$, we can find some iterate $z'$ of $z$ contained in $C\cap O$ and thus some cube $C'\subset C\cap O$ (containing $z'$) such that
$(C')^\circ\cap K$ is again nonempty and $C'\cap K$ is similar to $[0,1]^{k}\times P'$ for some set $P'\subset\R^{d-k}$.
Now let $y\in K$ and $\delta>0$ be arbitrary. Then there is some iterate $\varphi_\omega(C')$ with $\omega\in\Sigma_N^*$ of $C'$ contained in $B(y,\delta)$. Clearly, $\varphi_\omega(C')$ is a closed cube. Moreover, $C'\subset O$ implies $\varphi_\omega(C')\cap K=\varphi_\omega(C')\cap\varphi_\omega(K)=\varphi_\omega(C'\cap K)$. But the latter set is obviously similar to $C'\cap K$ and thus to $[0,1]^{d-k}\times P'$. This shows the local $k$-flatness of $K$.
\end{proof}

The following simple observation indicates that a locally flat self-similar set satisfying OSC is \emph{almost} globally flat in the sense that it is contained in a nontrivial product set.

\begin{prop}\label{globalflat}
Let $K\subset\R^d$ be a self-similar set satisfying OSC which is locally $k$-flat.
Let $C$ be any cube as guaranteed by the local $k$-flatness of $K.$
Then there is a similarity $\psi$ such that $K\subset\psi(C\cap K)$, that is, $K$ is contained in a $k$-flat set.
\end{prop}

\begin{proof}
We fix a strong feasible open set $O$ for $K$ and assume without loss of generality that $C\subset O$. (If this is not satisfied, one can work with a subcube $C'\subset C$ such that $C'\subset O$, as in the proof of Proposition~\ref{prop:one-point-flat}. If there is a $\psi$ such that $K\subset\psi(C'\cap K)$, then obviously the same $\psi$ also works for $C$.)

Choose $x\in C^{\circ}\cap K$ and put $D:=\dist(x,C^{c}).$
Let $\varphi_{1},...,\varphi_{N}$ be similarities generating $K$ with ratios $r_{1},...,r_{N}.$
If we put $r_{\max}=\max_{i=1,...,N} r_{i}<1$ we can find $m\in\N$ such that $r_{\max}^{m}\diam K<D.$
Let $\omega\in \Sigma_N^{m}$ such that $x\in\varphi_{\omega}(K)$, then $\varphi_{\omega}(K)\subset C^{\circ}.$
Hence for the similarity $\psi:=\varphi_{\omega}^{-1}$ the desired implication is satisfied.
\end{proof}

Note that the Proposition above only shows that a locally $k$-flat self-similar set is \emph{contained} in a (globally) $k$-flat set. It is not necessarily (globally) $k$-flat itself, as the following example illustrates:

\begin{ex} \label{ex:locallyflat_but_not_globally} \em
We start with the same mappings $\varphi_{k,l}:\R^2 \to \R^2$ as in Example~\ref{ex:1}.
Let $t_1,...,t_n\in\er$ and define for $k,l=1,...,n$
\begin{equation*}
\hat\varphi_{k,l}(x,y)=\varphi_{k,l}(x,y+t_k).
\end{equation*}
Fix $0<m<n$ and let $\hat K^{n,m}$ and $\hat F^{n,m}$ be the self-similar sets generated by the mappings $\hat\varphi_{k,l},$ $k=1,...,m$, $l=1$ and $\hat\varphi_{k,l},$ $k=1,...,m$, $l=1,...,n$, respectively; see~Figure~\ref{fig:ex1} (right) for an illustration.
Then, similarly as in Example~\ref{ex:1}, one has the relation $\hat F^{n,m}=\hat K^{n,m}\oplus(\{0\}\times[0,1])$ (where $\oplus$ means Minkowski addition) and therefore $\hat F^{n,m}$ is globally $1$-flat if and only if $t_1=t_2=\dots=t_m.$
Since $\hat K^{n,m}$ lies on the graph of some continuous function $f:[0,1]\to\R$ with variable $x$, $\hat F^{n,m}$ is locally $1$-flat. (This is seen as follows: Fix some point $x\in \hat K^{n,m}$. By the continuity of $f$ at $x$ there is some $\delta>0$ such that $|f(x)-f(y)|\leq 1/3 $ for every $y\in [x-\delta,x+\delta]$. Then the set $\hat K^{n,m}\oplus(\{0\}\times[0,1])\cap
([x-\delta,x+\delta]\times[f(x)+1/3,f(x)+2/3])$ is (globally) 1-flat.
Therefore we can choose a subsquare $C$ of $[x-\delta,x+\delta]\times[f(x)+1/3,f(x)+2/3]$ such that condition (ii) of Proposition~\ref{prop:one-point-flat} holds for the set
$\hat K^{n,m}\oplus(\{0\}\times[0,1])=\hat F^{n,m}$. Since $\hat F^{n,m}$
is self-similar, its local flatness follows from Proposition~\ref{prop:one-point-flat}.)
\end{ex}

The property of locally $k$-flat sets of being contained in a (globally) $k$-flat set implies that the flat cubes $C$ in the definition of local flatness are all aligned. We make this observation precise only in the case $d=2$, since we need it only for this case later on and since in higher dimensions the corresponding statement would be more technical.

\begin{cor}\label{flatremark}
Let $K\subset\R^2$ be a self-similar set satisfying OSC which is locally $1$-flat but not locally $2$-flat.
Let $C_{1}$ and $C_{2}$ be two different cubes obtained from the local flatness and let $L_{1}$ and $L_{2}$ be non-degenerate line segments in $C_{1}\cap K$ and $C_{2}\cap K,$ respectively.
Then $L_{1}$ and $L_{2}$ are parallel.
\end{cor}

\begin{proof}
The local $1$-flatness means that $C_{i}\cap K$ is a union of translates of $L_{i}$, for $i=1,2.$
Now, by Proposition~\ref{globalflat}, there is a similarity $\psi$ such that $K\subset \psi(C_{1}\cap K).$
This means that $C_{1}\cap K, C_{2}\cap K\subset \psi(C_{1}\cap K)$
and so in particular $L_{1}, L_{2}\subset \psi(C_{1}\cap K)$. Hence $L_1$ and $L_2$ must be parallel, because otherwise $K$ would be $2$-flat.
\end{proof}

We add another simple observation concerning the relation between local flatness of different orders:

\begin{lem}
Let $K\subset\er^{d}$ be a locally $k$-flat self-similar set and suppose that $C$ is a cube as in the definition of the local $k$-flatness with corresponding set $P$.
If $P$ is locally $1$-flat, then $K$ is locally $(k+1)$-flat.
\end{lem}

\begin{proof}
Suppose that $P$ is $1$-flat.
Choose $x\in C\cap K\cap O$ where $O$ is some feasible open set for $K.$
By Proposition~\ref{prop:one-point-flat}, it is sufficient to find $\eps>0$ and a cube $C'\subset B(x,\eps)$
such that $(C')^\circ\cap K$ is nonempty and $C'\cap K$ is similar to $[0,1]^{k+1}\times Q$ for some $Q\subset\er^{d-k-1}$.

Without loss of generality, we can assume that $C\cap K=[0,1]^{k}\times P$. Let $\pi$ be the orthogonal projection onto the last $d-k$ coordinates.
Then $x_0:=\pi(x)\in P$ and, since $P$ is locally $1$-flat, we can find $\eps>0$ and a cube $D\subset B(x_0,\eps)\subset\R^{d-k}$ such that $x_0\in D^{\circ}$ and  $D\cap P$ is similar to $[0,1]\times Q$
for some $Q\subset\er^{d-k-1}$.
Now, the desired cube $C'$ can be found around any point in $\er^{k}\times D\cap B(x,\eps)\cap K.$
\end{proof}

The following statement resolves the situation for full-dimensional sets.
\begin{thm}\label{fulldimension}
Let $K$ be a self-similar set in $\er^{d}$ satisfying OSC. Then the following assertions are equivalent:
\begin{enumerate}
  \item[(i)] $\dim_M K =d$,
  \item[(ii)] $s_{d-1}(K)<\dim_M K$,
  \item[(iii)] $K$ is locally $d$-flat.
\end{enumerate}
\end{thm}
\begin{proof} Since a self-similar set $K\subset\R^d$ of dimension $d$ has interior points (see \cite[Cor. 2.3]{S}) and since the interior points are dense in $K$, it is locally $d$-flat.
On the other hand, any locally $d$-flat set contains an open set and has thus obviously dimension $d$. This proves (i)$\Leftrightarrow$(iii). The implication (ii)$\Rightarrow$(i) follows by contraposition from the fact that $\dim_M K<d$ implies $s_{d-1}(K)=\overline{\dim}_S K=\overline{\dim}_M K$ for any bounded set $K\subset\R^d$, cf. \cite[Corollary 3.6]{rw09} or \cite[Theorem 1.1]{rw11}.

It remains to prove the implication (i)$\Rightarrow$(ii), which will follow at once if we show that $\udim_M \partial K<d.$ Indeed, if this strict inequality holds, then, by \cite[Corollary 3.6]{rw09}, we have
$$
     d>\udim_M \bd K=\udim_S \bd K\geq \udim_S K=s_{d-1}(K),
$$
where the second inequality is due to the set inclusion
$
\bd (K_r)=\bd((\bd K)_r)\cap K^c\subseteq \bd((\bd K)_r),
$
which holds for each $r>0$.

For a proof of the inequality $\udim_M \partial K<d$, observe that $\dim_M K=d$, implies the interior of $K$ is nonempty and we can choose $x\in K^{\circ}.$
Set $D=\dist(x,\partial K)>0.$
Let $\varphi_{1},...,\varphi_{N}$ be similarities generating $K$ with ratios $\rho_{1},...,\rho_{N}.$
Putting $\rho:=\max_{i=1,...,N} \rho_{i}<1$, we can find $m\in\en$ such that $\rho^{m}\diam K<D.$
Let $\psi_{1},...,\psi_{k}$ be all mappings of the form $\varphi_{\omega}$ with $\omega\in\{1,...,N\}^{m}$ ordered in such a way that $x\in\psi_{k}(K).$
Then the IFS $\{\psi_{1},...,\psi_{k}\}$ also generates $K$, satisfies OSC and moreover, $\psi_{k}(K)\subset K^{\circ}$.

Suppose that $r_{i}$ is the contraction ratio of $\psi_{i}.$
Since $\dim_M K=d$, we have $\sum_{i=1}^{k} r_{i}^{d}=1.$
Let $L$ be the self-similar set generated by the mappings $\psi_{1},...,\psi_{k-1}$.
Since $U:=K^{\circ}$ is a feasible open set for $K$ (with $\cl{U}=K$, see e.g.~\cite[Proposition~5.4]{PW}, and thus $\bd U=\bd K$) and by the choice of $\psi_k$ we have
$$
\partial U\subset \bigcup_{i=1}^{k-1}\psi_i (\partial U).
$$
Similarly, for any $n\in\N$
$$
\partial U\subset \bigcup_{|\omega|=n}\psi_\omega (\partial U)\subset \bigcup_{|\omega|=n}\psi_\omega (\overline U)=:L_n,
$$
where the unions are over all words $I\in\{1,2,\ldots,k-1\}^n$.
Therefore, using the fact that $L_n\to L$ in the Hausdorff metric as $n\to \infty$, we obtain that $\partial K=\partial U\subset L.$
Since $\sum_{i=1}^{k-1} r_{i}^{d}<1$ we have $\dim_M L<d$ and therefore $\dim_M\partial K<d$ as claimed.
\end{proof}

Theorem~\ref{fulldimension} is sufficient to resolve Conjecture~\ref{conj} for sets in $\R$.
\begin{cor} \label{cor:main-R}
  Let $K\subset\R$ be a self-similar set satisfying OSC. Then $s_0(K)<\dim_M K$ if and only if $K$ is locally flat.
In this case, $\dim_M K=1$.
\end{cor}

 We conjecture that Theorem~\ref{fulldimension} can be generalized to sets in $\R^d$ that are full-dimensional with respect to their affine hull. That is, for any set $K\subset\R^d$ whose affine hull has dimension $n\in\{1,\ldots,d\}$, the assertions (i)--(iii) in Theorem~\ref{fulldimension} are equivalent if $d$ is replaced with $n$. In fact, it is rather obvious that (i) and (iii) are equivalent, as the concept of local $k$-flatness is independent of the dimension of the ambient space as is the notion of Minkowski dimension (see \eqref{eq:s-exponents} below).
 To show the equivalence of (ii) with (i) (and (iii)), it seems however necessary to prove that scaling exponents of a set are independent of the dimension of the ambient space. We discuss this independence now for the case $n=1$, which is the only case we need to resolve Conjecture~\ref{conj} in the plane. The case of a general $n$ seems more difficult.

 In order to formulate the problem precisely, it is necessary to extend the notation to be able to distinguish different ambient space dimensions.
Recall from \eqref{eq:sk-def1} that the definition of scaling exponents of a set $F$ is based on parallel sets which depend on the choice of the ambient space. For instance, for a subset $F$ of $\R$, the parallel set in $\R$ is a finite union of segments, while the parallel set of $F$ in $\R^2$ is a two-dimensional set. Up to now, we have not emphasized this dependence. For a subset $F\subset\R^d$, we always considered the full parallel set in $\R^d$.

We will now use an extra upper index to indicate the dimension of the parallel set. For a set $F\subset\R^d$ with $\dim\aff F=n\geq 1$ and $l\in\{n,\ldots, d\}$, we write $F^l_\eps$ for the $\eps$-parallel set of $F$ in $\R^l$. (Note that $F_\eps^l$ is independent of the choice of the embedding space. Any $l$-dimensional subspace of $\R^d$ containing $F$ (and thus $\aff F$) provides the same parallel set up to isometry.)  We write $s_k^l(F)$ for the $k$-th scaling exponent of $F$ based on the $l$-dimensional parallel sets of $F$. Note that this notation makes only sense for $n\leq l\leq d$ and $k\leq l$.
Similarly, we use $\dim_M^l F$ and $\dim_S^l F$ to indicate the dimension dependence.
In this notation, we have, by definition, the relations
$$
s_l^l(F)=\udim_M^l F \text{ and } s_{l-1}^l(F)=\udim_S^l F
$$
for any $l\in\{n,\ldots,d\}$, whenever the scaling exponents are defined.
The well-known independence of the (upper) Minkowski dimension on the dimension of the ambient space
is then described by the relation
\begin{align}\label{eq:s-exponents}
   s_n^n(F)=s_{n+1}^{n+1}(F)=\ldots=s_d^{d}(F).
\end{align}
By application of \cite[Corollary~3.6]{rw09} to different ambient space dimensions, it follows immediately that
\begin{align} \label{eq:s-exponents2}
   s_{l-1}^l(F)=s_{l}^{l}(F)
\end{align}
for $l\in\{n+1,\ldots,d\}$. Equation \eqref{eq:s-exponents2} also holds for $l=n$, provided $\H^n(F)=0$ (or provided the upper Minkowski dimension $s_n^n(F)$ is replaced by the upper outer Minkowski dimension, see \cite[paragraph before Cor. 3.4]{rw09}).

In general (that is, for any bounded set $F\subset\R^d$ with $\dim \aff F=n\geq 1$), we conjecture that the relation
\begin{align} \label{eq:conj-exp}
   s_k^l(F)=s_k^{l+1}(F)
\end{align}
holds for any $k\in\{0,\ldots,d\}$ and any $l\in\{n,\ldots,d-1\}$ such that $l\geq k$, provided the exponents are well defined.
The first interesting case is $d=2$ and $n=1$ (for $d=1$ there are no such relations), for which the two relations  $s_0^1(F)=s_0^2(F)$ and $s_1^1(F)=s_1^2(F)$ are conjectured. This case is resolved in Proposition~\ref{subsets_of_R} below, which is another important step towards the resolution of Conjecture~\ref{conj} in $\R^2$.
In general, we note that, by combining the relations \eqref{eq:s-exponents} and \eqref{eq:s-exponents2} above, one gets immediately for each $k\in\{n,\ldots,d-1\}$,
\begin{align} \label{eq:conj-resol}
   s_k^k(F)=s_{k+1}^{k+1}(F)=s_{k}^{k+1}(F)
\end{align}
(and these exponents are always well defined) --  resolving the case $k=l$ of \eqref{eq:conj-exp}.

\begin{figure}
\begin{minipage}{120mm}
  \includegraphics[width=120mm]{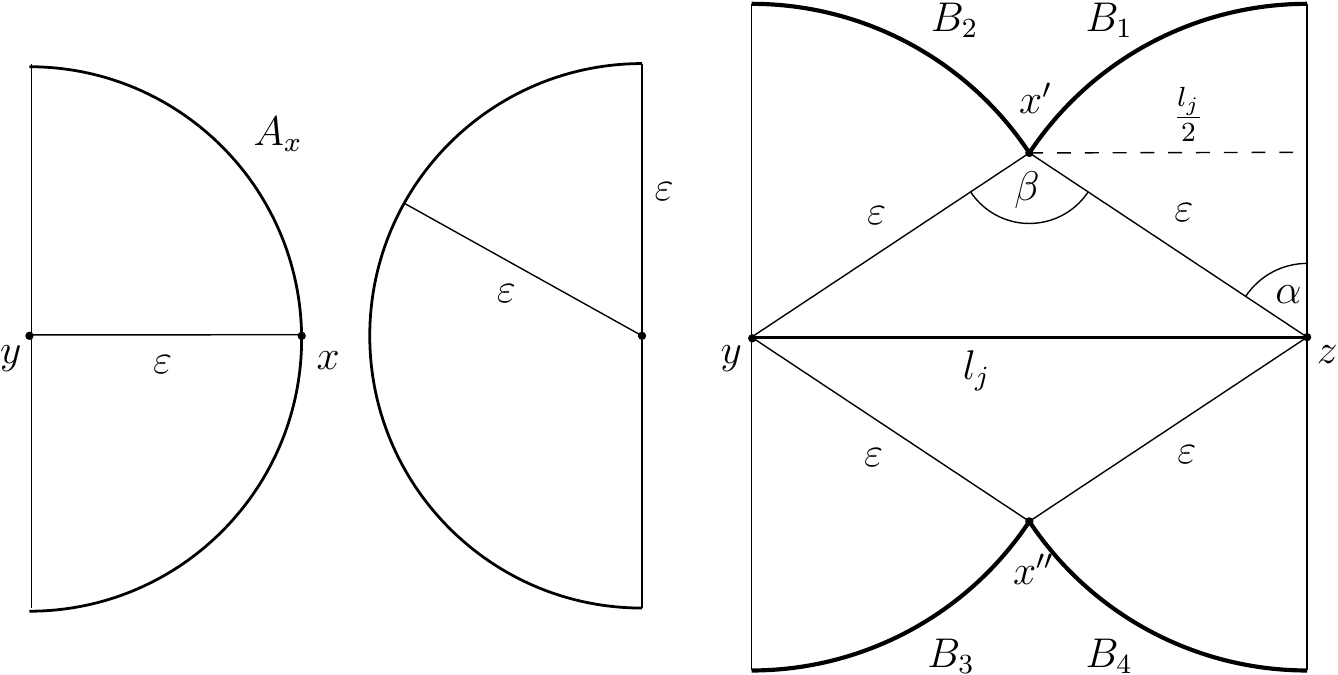}
 \end{minipage}
\caption{\label{fig:curv_arc} The set $F_\eps^2\cap (I_j\times \R)$ in the case $l_j> 2\eps$ (left) and in the case $l_j\leq 2\eps$ (right).}
\end{figure}

\begin{prop}\label{subsets_of_R}
Let $F\subset\R^d$ be a bounded set with $\dim \aff F=1$. Then
$$
s_0^1(F)=s_0^2(F)=\ldots =s_0^d(F).
$$
\end{prop}
\begin{proof}
Since it makes no difference whether the parallel set $F_\eps^k$ is studied in $\R^k$ or in $\R^m$ with $m> k$, it suffices to prove that $s_0^1(F)=s_0^d(F)$ in $\R^d$ for any $d\geq2$.

First we will show the inequality $s_0^1(F)\leq s_0^d(F)$, for which we employ the \emph{fractal string} $\sL=(l_j)_{j\in\N}$ associated to $F$, that is, the sequence of the lengths $l_j$ of the complementary intervals $I_j$ of $\overline{F}$ ordered in a non-increasing way. Note that $s_0^1(F)=\udim_S^1 F$ and that the corresponding curvature measure is the counting measure on $\bd(F_\eps^1)$ (recall, $F_\eps^1$ is the $\eps$-parallel set in $\R$), i.e., $C_0^\var(F_\eps^1)=C_0(F_\eps^1)=\frac 12 \H^0(\bd(F_\eps^1))$. The latter is given in terms of $\sL$ by
$$
\H^0(\bd(F_\eps^1))=2+2\# \{j: l_j>2\eps\}.
$$
Observe that to each point $x\in\bd(F_\eps^1)$ there is a unique nearest point $y\in F$ (with $d(x,y)=\eps$) either to the left or to the right of $x$, i.e. $y=x+\eps$ or $y=x-\eps$.
Moreover, to each such $x$ there corresponds a unique $(d-1)$-dimensional half-sphere $A_x$ in $\bd(F_\eps^d)$ with radius $\eps$ and centre $y$. (In the case $d=2$, $A_x$ is a half-circle, cf.~Figure~\ref{fig:curv_arc} (left).)  Obviously, $C^\var_0(F_\eps^d, A_x)=\frac 12$ for each $x\in\bd(F_\eps^1)$ and in any dimension $d\geq 2$.
Therefore,
$$
C^\var_0(F_\eps^1)=\frac 12 \H^0(\bd(F_\eps^1))=\sum_{x\in\bd(F_\eps^1)} \frac 12=\sum_{x\in\bd(F_\eps^1)}C^\var_0(F_\eps^d, A_x)\le C^\var_0(F_\eps^d),
$$
from which the inequality $s_0^1(F)\leq s_0^d(F)$ follows immediately. (For each $t> s_0^d(F)$, one has $\lim_{\eps\to 0} \eps^t C_0^\var(F_\eps^d)=0$ and thus $\lim_{\eps\to 0} \eps^t C_0^\var(F_\eps^1)=0$ by the above inequality, implying $s_0^1(F)\leq t$.)

   The reverse inequality is now first proved for the case $d=2$, for which  we split the parallel set $F_\eps^2$ as follows. Denoting by $I$ the smallest closed interval containing $F$, we have the disjoint composition
$$
\R^2=(F\times\R)\cup (I^c\times \R) \cup \bigcup_{j=1}^\infty (I_j\times \R)
$$
and thus
\begin{align}\label{eq:parallel_decomposed}
   C^\var_0(F_\eps^2)\le C^\var_0(F_\eps^2,F\times\R)+C^\var_0(F_\eps^2, I^c\times \R)+\sum_{j=1}^\infty C^\var_0(F_\eps^2,I_j\times \R).
\end{align}
It is not difficult to see that the first term in this sum is zero and the second term is 1 (curvature of two half circles). For the terms in the remaining sum, we distinguish between those $j$ for which $l_j>2\eps$ and those for which $l_j\leq 2\eps$ holds.
In the first case, $\bd(F_\eps^2)\cap (I_j\times R)$ consists of two disjoint half-circles of radius $\eps$ (see Figure~\ref{fig:curv_arc} (left)) such that $C^\var_0(F_\eps^2,I_j\times \R)=1$. In the second case, one has
$$
C^\var_0(F_\eps^2,I_j\times \R)=\frac 4\pi \arcsin(\frac{l_j}{2\eps}),
$$
which is easily seen from Figure~\ref{fig:curv_arc} (right). Indeed, we have positive curvature $\alpha/(2\pi)$ on each of the 4 arcs $B_i$ and negative curvature $\beta/(2\pi)$ at the points $x'$ and $x''$ where the arcs meet. Since $\beta=2\alpha$, we thus obtain $C^\var_0(F_\eps^2,I_j\times \R)=4\cdot \frac \alpha {2\pi}+ 2\cdot \frac\beta{2\pi}=\frac{4}\pi \alpha$. Now the claim follows by noting that the angle $\alpha$ is determined by the relation $\sin \alpha=\frac{l_j}{2\eps}$.
Using that $\arcsin(t)\leq \frac\pi 2 t$ for $t\in[0,1]$, we infer that
$$
C^\var_0(F_\eps^2,I_j\times \R)\leq \frac{l_j}{\eps},
$$
for each $j$ such that $l_j\leq 2\eps$.
Plugging all this into equation \eqref{eq:parallel_decomposed}, we get
\begin{align*} 
    C^\var_0(F_\eps^2)&\le 1+\sum_{j:l_j>2\eps} 1 + \eps^{-1} \sum_{j:l_j\leq 2\eps} l_j\\
    &\leq \eps^{-1}\left(2\eps + \sum_{j:l_j>2\eps} 2\eps + \sum_{j:l_j\leq 2\eps} l_j\right).
\end{align*}
Now observe that the expression in parentheses is exactly the length of the parallel set $F_\eps^1$. Therefore,
$$
 C^\var_0(F_\eps^2)\le \eps^{-1} \lambda_1(F_\eps^1\setminus F).
$$
Now let $t>s_0^1(F)$. Since $s_0^1(F)=\udim_S^1 F$ coincides with the upper outer Minkowski dimension,  see equation \eqref{eq:s-exponents2},
we have
$
\lim_{\eps\to 0} \eps^{t-1} \lambda_1(F_\eps^1\setminus F)=0,
$
which implies
$
\lim_{\eps\to 0} \eps^{t}\ C^\var_0(F_\eps^2)=0
$
and thus $s_0^2(F)\geq t$. This shows $s_0^1(F)\geq s_0^2(F)$ and the proof for the case $d=2$ is complete. If $d$ is arbitrary, for the proof of $s_0^1(F)\geq s_0^d(F)$, one can decompose $\R^d$ similarly as $\R^2$ above and obtain an estimate for $C_0^\var(F_\eps^d)$ similar to equation \eqref{eq:parallel_decomposed}:
\begin{align*}
   C^\var_0(F_\eps^d)\le C^\var_0(F_\eps^2,F\times\R^{d-1})+C^\var_0(F_\eps^2, I^c\times
   \R^{d-1})+\sum_{j=1}^\infty C^\var_0(F_\eps^2,I_j\times \R^{d-1}).
\end{align*}
 It is easy to see that also in this general situation the first term vanishes while the second term is equal to 1. In the remaining sum, for all indices $j$ such that $2\eps<l_j$, we still have $C_0^\var(F_\eps^d,I_j\times \R^{d-1})=1$ since this set has the curvature of a $d$-dimensional ball. For all $j$ such that $2\eps\geq l_j$, we claim that
\begin{align}\label{eq:case-d}
C^\var_0(F_\eps^d,I_j\times \R^{d-1})\leq c_d \frac{l_j}{\eps},
\end{align}
for some constant $c_d$ independent of $j$ and $\eps$. Then all the remaining arguments carry over from the case $d=2$ discussed above.

For a proof of \eqref{eq:case-d}, let $y$ and $z$ denote the endpoints of $I_j$. Note that $C_0^\var(F_\eps^d, I_j\times \R^{d-1})=C_0^\var(B^d(y,\eps)\cup B^d(z,\eps), I_j\times \R^{d-1})$, i.e., we have to compute the curvature of two intersecting $\eps$-balls. Recalling that a ball as well as the union of two intersecting balls have total curvature 1, the additivity and symmetry yield that
$$
C_0^\var(F_\eps^d, I_j\times \R^{d-1})=4\cdot C_0(B^d(y,\eps), [y,\frac{y+z}2)\times \R^{d-1})=4\cdot C_0(B^d(0,\eps), [0,\frac{l_j}2)\times \R^{d-1})
$$
Now observe that, by symmetry, the curvature of any subset $S$ of the boundary of $B^d(0,\eps)$ is given by the normalized volume of the associated cone $K(S):=\bigcup\{[0,s]: s\in S\}$, that is,
\begin{align*}
C_0(B^d(0,\eps), S)=\frac{V_d(B^d(0,\eps)\cap K(S))}{V_d(B^d(0,\eps))}=\frac{V_d(B^d(0,\eps)\cap K(S))}{\alpha_d \eps^{d}}.
\end{align*}
Here $\alpha_n$ denotes the volume of the $n$-dimensional unit ball.
Since for $S=\bd B^d(0,\eps)\cap ([0,l_j/2)\times \R^{d-1})$ the cone $K(S)$ is contained in the cylinder $Z=[0,\frac{l_j}{2}]\times B^{d-1}(0,\eps)$, whose volume is given by $\frac{l_j}{2} \alpha_{d-1}\eps^{d-1}$, we get
\begin{align*}
C_0^\var(F_\eps^d, I_j\times \R^{d-1})&=4 \frac{V_d(B^d(0,\eps)\cap K(S))}{V_d(B^d(0,\eps))}
\leq  \frac{4\, V_d(Z)}{\alpha_d \eps^{d}}\\
&=  \frac{4 \frac{l_j}{2} \alpha_{d-1}\eps^{d-1}}{\alpha_d \eps^{d}}
=\frac{2\alpha_{d-1}}{\alpha_d} \frac{l_j}{\eps},
\end{align*}
 This proves the estimate in \eqref{eq:case-d} for the constant $c_d:=2\alpha_{d-1}/\alpha_d$ and completes the proof of the inequality $s_0^d(F)\leq s_0^1(F)$.
\end{proof}

The following statement is a generalization of Theorem~\ref{fulldimension} to sets in $\R^d$ with a 1-dimensional affine hull. We will apply it later in particular in the case $d=2$.
\begin{prop}\label{embedding fulldimension}
Let $K$ be a self-similar set in $\er^{d}$ satisfying OSC and assume that the affine hull of $K$ has dimension $1$. 
Then the following assertions are equivalent:
\begin{enumerate}
  \item[(i)] $\dim_M K =1$,
  \item[(ii)] $s_{0}(K)<\dim_M K$,
  \item[(iii)] $K$ is locally $1$-flat.
\end{enumerate}
\end{prop}
\begin{proof}
It is easy to see that $K$ viewed as a set in $\R$ is locally $1$-flat if and only if $K$ is locally 1-flat as a set in $\R^d$. (Indeed, if $C\subset \R^d$ is a flat cube for $K$ in $\R^d$ for some $x\in K$ and $\eps>0$, then its projection $C'$ onto $\R\times\{0\}^{d-1}$ is a flat cube for $K$ and vice versa.) Therefore the equivalence of (i) and (iii) follows immediately from Theorem~\ref{fulldimension} and the fact that die Minkowski dimension is independent of the dimension of the ambient space (see \eqref{eq:s-exponents}). The equivalence of (ii) and (iii) is also direct consequence of Theorem~\ref{fulldimension} taking into account the relation $s_0^d(K)=s_0^1(K)$
derived in Proposition~\ref{subsets_of_R}.
\end{proof}

\section{Results for self-similar sets in $\R^d$} \label{sec:s0}

   Now we discuss some simple geometric conditions for self-similar sets in $\R^d$ which ensure that their $0$-th scaling exponents are equal to their dimension. More precisely, we will show that this is true for all self-similar sets whose complement is disconnected (see Theorem~\ref{thm:conn-compl}) and for all sets that are totally disconnected (see Corollary~\ref{cor:disconn}).
   Throughout we assume that $K\subset\R^d$ is a \emph{regular} self-similar set, by which we mean that almost all $\eps>0$ are regular for $K$, cf.\ Section~\ref{sec:pre}.

   The following observation is essential for the results in this section. For a set $A\subset\R^d$ with positive reach, we denote by $\nor A$ its \emph{normal bundle} and for $(x,n)\in\nor A$, $\Tan(\nor A,(x,n))$ is the tangent cone of $\nor A$ at $(x,n)$. Let $\omega_d:=\Ha^{d-1}(S^{d-1})$ be the surface area of the unit sphere $S^{d-1}$ in $\R^d$, and let $\pi_{2}:\er^{d}\times S^{d-1}\to S^{d-1},(x,n)\mapsto n$ be the projection onto the second component.

\begin{lem}\label{lem:C0var}
Let $A$ be a set with positive reach, $B\subset\partial A$ and $S\subset S^{d-1}.$
Define $B^{*}=\{(x,n)\in\nor A: x\in B\}$ and assume that $\pi_{2}(B^{*})\supset S.$
Then
\begin{align}\label{eq:C0var}
  \gC_{0}^{\var}(A,B^{*})\geq \omega_d^{-1} \Ha^{d-1}(S).
\end{align}
In particular, if $A$ is compact, then $\gC_{0}^{\var}(A)\geq 1$. Similarly, if the closed complement $\widetilde{A}$ of $A$ is bounded (and still, $A$ has positive reach), then $\gC_{0}^{\var}(A)\geq 1$.
\end{lem}

\begin{proof}
Due to \cite{Z1}, for $\H^{d-1}$-almost all $(x,n)\in\nor A$, $\Tan(\nor A, (x,n))$ is a $(d-1)$-dimensional linear space and
orthonormal principal directions $a_{1}(x,n),\ldots,$ $a_{d-1}(x,n)\in\er^{d}$ as well as
the corresponding (generalized) principal curvatures $\kappa_{1}(x,n),...,$ $\kappa_{d-1}(x,n)\in(-\infty,\infty]$
are well defined. The vectors
$$
\left(\frac{a_{i}(x,n)}{\sqrt{1+\kappa_{i}^{2}(x,n)}},\frac{\kappa_{i}(x,n)a_{i}(x,n)}{\sqrt{1+\kappa_{i}^{2}(x,n)}}\right),\quad i=1,...,d-1,
$$
form an orthonormal basis of $\Tan(\nor(A), (x,n)).$
Therefore
\begin{equation}\label{valuedet}
\sqrt{\det\left(D\pi_2(x,n)\cdot (D\pi_2(x,n))^{*}\right)}=\left|\frac{\prod_{i=1}^{d-1}\kappa_{i}(x,n)}{\prod_{i=1}^{d-1}\sqrt{1+\kappa_{i}^{2}(x,n)}}\right|
\end{equation}
for $\H^{d-1}$-almost every $(x,n)\in\nor A$.
On the one hand, using Federer's coarea formula \cite[{§} 3.2.22]{F69}, we get
\begin{equation}\label{measure}
\begin{aligned}
\int_{B^{*}}& \sqrt{\det\left(D\pi_2(x,n)\cdot (D\pi_2(x,n))^{*}\right)}\, d\H^{d-1}(x,n)\\
&=\int_{S^{d-1}}\H^{0}(B^{*}\cap \pi_2^{-1}(y))\,d\H^{d-1}(y)
\geq \int_{S}1\;d\H^{d-1}(y)=\H^{d-1}(S).
\end{aligned}
\end{equation}
On the other hand, by \cite[Theorem 3]{Z1}, we have
\begin{align*}
\gC_{0}^{\var}(A,B^{*})=\omega_d^{-1}\int_{B^{*}}\left|\frac{\prod_{i=1}^{d-1}\kappa_{i}(x,n)}{\prod_{i=1}^{d-1}\sqrt{1+\kappa_{i}^{2}(x,n)}}\right|d\H^{d-1}(x,n).
\end{align*}
Now it suffices to combine this with (\ref{valuedet}) and (\ref{measure}) to obtain \eqref{eq:C0var}.

If $A$ is compact, choose $B=\bd A$ and observe that $\pi_2(B^*)=S^{d-1}$ (since for every direction $n\in S^{d-1}$ there is a hyperplane with normal vector $n$ supporting $A$ in at least one point $x$ and $(x,n)\in \nor A$). Hence one can choose $S=S^{d-1}$ and the second assertion follows from \eqref{eq:C0var}. If $\widetilde{A}$ is bounded, we can argue similarly. For $n\in S^{d-1}$, there is a hyperplane $H_n$ with normal direction $n$ touching $\tilde A$ in (at least) one point $x\in\bd \tilde A=\bd A$. Since $A$ has positive reach and $x\in\bd A$, there is at least one direction $n'\in S^{d-1}$ such that $(x,n')\in\nor A$. Because of the touching hyperplane $H_n$ which belongs to the tangent cone of $A$ at $x$, this normal direction $n'$ is unique and equal to $-n$. Since $n\in S^{d-1}$ was arbitrary, we infer that $\pi_2((\bd A)^*)=S^{d-1}$. Hence we can again apply \eqref{eq:C0var} with $S=S^{d-1}$ and conclude that $\gC_{0}^{\var}(A)\geq 1$. This completes the proof.
\end{proof}

The following statement is a general scheme to show that the (similarity) dimension of a self-similar set is a lower bound for its scaling exponents. It is a modification and generalization of \cite[Theorem 2.3.8]{W1}. We have formulated a version for curvature direction measures; a corresponding statement holds for the curvature measures i.e., if in the hypothesis as well as in the conclusion $\gC_k^\var(K_\eps,\cdot)$ is replaced by $C_k^\var(K_\eps,\cdot)$.
Compared to Theorem~2.3.8 in \cite{W1}, the polyconvexity assumption is weakened to regularity. Moreover, alternative to the assumption $B\subset O_{-\eps_0}$ (with $O_{-\eps}:=((O^c)_\eps)^c$ being the inner $\eps$-parallel set of $O$), also the assumption $\bd B\subset K$ allows the same conclusion.
We point out that no integrability assumptions for the curvature measures are required for this statement.

\begin{prop} \label{prop:w1-2-3-8}
Let $K\subset\R^d$ be a regular self-similar set satisfying OSC, $O$ some feasible open set of $K$, $D:=\dim_M K$ and $k\in\{0,\ldots,d\}$.
Suppose
there exist some constants $\eps_0,\beta>0$ and some open set $B\subset O$, satisfying at least one of the conditions $\bd B\subset K$ or $B\subset O_{-\eps_0}$, such that, for almost all $\eps\in(r_{\min}\eps_0,\eps_0]$,
\[\gC_k^\var(K_\eps, B\times S^{d-1})\ge\beta.\]
Then, for almost all $\eps<\eps_0$,
\[\eps^{D-k} \gC_k^\var(K_\eps)\ge c ,\]
where $c:=\beta \eps_0^{D-k} r_{\min}^D>0\,$.
In particular, it follows $\tilde s_k(K)\geq\dim_M K$.
\end{prop}

\begin{proof}
Due to the new hypothesis, the proof of Theorem~2.38 in \cite{W1} needs some adaptations, although the essential argument carries over. Let $S_1,\ldots,S_N$ be similarities generating $K$.
Since for each $r>0$ the sets $S_\omega O$, $\omega\in\Sigma(r):=\{\sigma\in\Sigma_N^*: r_\sigma<r\le r_\sigma r_{\min}^{-1}\}$, are pairwise disjoint, the same holds for their subsets $S_\omega B$ and so, for regular $\eps>0$,
\[ \gC_k^\var(K_\eps)\ge \sum_{\omega\in\Sigma(r)} \gC_k^\var(K_\eps, S_\omega B\times S^{d-1}).\]
Fix some regular $\eps<\eps_0$ and set $r:=r_{\min}^{-1}\eps_0^{-1}\eps$. First, we claim that for each $\omega\in\Sigma(r)$,
\begin{align} \label{eq:loc-in-B}
  K_\eps\cap S_\omega B= (S_\omega K)_\eps\cap S_\omega B.
\end{align}
For a proof in the case when $\bd B\subset K$, let $x\in K_\eps\cap S_\omega B$. Then there exists a point $y\in K$ such that $d(x,y)=d(x,K)\leq\eps$. Since $\bd B\subset K$, one has either $y\in S_\omega \bd B \subset S_w K$ or $y\in S_\omega B\cap K$. The latter also implies $y\in S_\omega K$, since otherwise one would have $y\in S_\sigma K$ for some $\sigma\in\Sigma(r)$, $\sigma\neq \omega$ and so $S_\sigma K\cap S_\omega O\neq \emptyset$, which violates OSC.
Thus in both cases $y\in S_\omega K$, which implies $x\in (S_\omega K)_\eps$, proving one inclusion in \eqref{eq:loc-in-B} in the case when $\bd B\subset K$. The reverse inclusion is obvious, since $S_\omega K\subset K$. In case $B\subset O_{-\eps_0}$, equation~\eqref{eq:loc-in-B} follows from the relation
$K_\eps\cap (S_\omega O)_{-\eps}= (S_\omega K)_\eps\cap (S_\omega O)_{-\eps}$ for each $\eps\leq\eps_0$.
Equation \eqref{eq:loc-in-B} allows to use the locality property, which together with the scaling property of $\gC_k^\var$ yields
\[
\gC_k^\var(K_\eps, S_\omega B\times S^{d-1})=\gC_k^\var(S_\omega(K_{\eps r_\omega^{-1}}), S_\omega B\times S^{d-1})=r_\omega^k \gC_k^\var(K_{\eps r_\omega^{-1}}, B\times S^{d-1}).
\]
Since, by the choice of $r$, we have $\eps r_\omega^{-1} \in (r_{\min}\eps_0,\eps_0]$, the hypothesis implies that $\gC_k^\var(K_{\eps r_\omega^{-1}}, B\times S^{d-1})\ge\beta$ and
therefore,
\[
\gC_k^\var(K_\eps)\ge \sum_{\omega\in\Sigma(r)} r_\omega^k \gC_k^\var(K_{\eps r_\omega^{-1}}, B\times S^{d-1})\ge \sum_{\omega\in\Sigma(r)} (\eps_0^{-1} \eps)^k \beta=\beta \eps_0^{-k} \eps^k \card{\Sigma(r)}.
\]
Recalling that $\card{\Sigma(r)}\ge r^{-D}=r_{\min}^{D}\eps_0^{D}\eps^{-D}$, see e.g.\ \cite[eq.~(5.1.5)]{W1},
we obtain
\[ \gC_k^\var(K_\eps)\ge \beta r_{\min}^{D} \eps_0^{D-k}  \eps^{-D+k}=c \eps^{k-D}\]
as claimed, completing the proof of Proposition~\ref{prop:w1-2-3-8}.
\end{proof}

As an application of Proposition~\ref{prop:w1-2-3-8}, we will now formulate two simple geometric conditions each of which ensures generic behaviour of the $0$-th scaling exponent.

\begin{thm} \label{thm:conn-compl}
Let $K\subset\R^{d}$ be a regular self-similar set satisfying OSC with $\dim_M K<d$. Suppose that the complement $K^c$ of $K$ has a bounded connected component.
Then $\tilde s_{0}(K)\geq \dim_M K$.
\end{thm}

\begin{proof}[Proof of Theorem~\ref{thm:conn-compl}]
 Let $\varphi_1,\ldots,\varphi_N$ be similarities generating $K$. Fix a feasible open set $O$ for the SOSC.  First we claim that the existence of a bounded connected component $B$ of $K^c$ implies the existence of such a component $B'$ with $B'\subset O$.
  To see this, let $x\in K\cap O$. Choose $\omega\in\Sigma^*_N$ such that $\conv(\varphi_\omega(K))\subset O$. Since $B\subset \conv(K)$, we have $\varphi_\omega(B)\subset O$ and $\bd\varphi_\omega(B)=\varphi_\omega(\bd B)\subset \varphi_\omega(K)\subset K$. Hence $\varphi_\omega(B)$ is open and has its boundary contained in $K$, but it is not necessarily a connected component of $K^c$. However, since $K$ cannot fill the whole open set $B$, there must be at least one connected component $B'$ of $K^c$ contained in $\varphi_\omega(B)$ and thus in $O$ proving the claim.

 By the above claim it is justified to assume that $B\subset O$ in the sequel. The OSC implies that $\varphi_i(B)\cap \varphi_j(B)=\emptyset$ for $i\neq j$.
Let $\rho=\rho(B)$ be the inradius of the set $B$ and let $\eps\in(0,\rho)$ be a regular value for $K$. Then the closed complement $\widetilde{K_\eps}$ of $K_\eps$ has positive reach. It is easy to see that the set $A:=\cl{B\setminus K_\eps}=\widetilde{K_\eps}\cap B$ is a subset of $\widetilde{K_\eps}$ that is well separated from the rest of $\widetilde{K_\eps}$ (with distance at least $2\eps$).
Hence $A$ has positive reach. Since $B\cap K=\emptyset$ and $\bd B\subset K$, we have $K_\eps\cap B=(\bd B)_{\eps}\cap B$
and thus, by the locality property,
$\gC_0^\var(K_\eps, B\times S^{d-1})=\gC_0^\var((\bd B)_{\eps}, B\times S^{d-1})$. By the reflection principle, the latter expression equals
$\gC_0^\var(A, B\times S^{d-1})=\gC_0^\var(A)$. Since $A$ is a compact set with positive reach, we can apply the second part of Lemma~\ref{lem:C0var} to infer
$
\gC_0^\var(A)\geq 1.
$
Fixing some $\eps_0<\rho$, we have therefore
$$
\gC_0^\var(K_\eps, B\times S^{d-1})\ge 1
$$
for all regular values $\eps\in(0,\eps_0]$.  Hence, the hypotheses of Proposition~\ref{prop:w1-2-3-8} are satisfied and we conclude $\tilde s_0(K)\geq\dim_M K$, which completes the proof of Theorem~\ref{thm:conn-compl}.
\end{proof}

\begin{lem}\label{bigpieces}
Let $K$ be a regular self-similar set in $\R^{d}$ satisfying OSC such that $\tilde s_{0}(K)<\dim_M K$. Then there is an open set $U\subset\R^{d}$ and $\delta>0$ such that $U\cap K\not=\emptyset$
and every connected component of $K$ that intersects $U$ has diameter at least $\delta$.
Moreover, $U$ can be chosen as a subset of some feasible open set for $K$.
\end{lem}

\begin{proof}
Let $K$ be generated by $\varphi_{1},...,\varphi_{N}.$
For every $n\in \en$, we define an equivalence relation $\approx_{n}$ on $\Sigma_N^{n}$ as follows: For $\omega,\sigma\in\Sigma_N^{n}$, $\omega\approx_{n} \sigma$ if and only if
there are $p\in\N$ and words $\omega=\omega_{1},\omega_{2},...,\omega_{p}=\sigma\in\Sigma_N^{n}$ such that $\varphi_{\omega_{l}}(K)\cap\varphi_{\omega_{l+1}}(K)\not=\emptyset$ for $l=1,..,p-1.$

Let $\Gamma_{i}^{n},$ $i=1,...,t_{n}$ be the equivalence classes of $\approx_{n}.$
 We will refer to the union sets $C^{n}_{i}=\bigcup_{\omega\in\Gamma^{n}_{i}}\varphi_{\omega}(K)$ as the \emph{clusters of level $n$}. Set $\eps_{n}=\frac{1}{3}\min_{j\not=i}\dist(C^{n}_{i},C^{n}_{j})>0.$
Then, by Lemma~\ref{lem:C0var}, we have, for each regular $0<\eps<\eps_{n}$ and each $i$,
\begin{align}\label{eq:cluster_curv}
\gC^{\var}_{0}(K_{\eps},(C^{n}_{i})_{\eps}\times S^{d-1})\geq 1.
\end{align}
Now, by SOSC, there is a feasible open set $O$, $x\in K$ and $\delta>0$ such that $B(x,3\delta)\subset O.$
Define $U=B(x,\delta).$
Suppose for contradiction that there is a connected component $C$ of $K$ such that $\diam C<\delta$ and $C\cap U\not=\emptyset.$
Then $C\subset B(x,3\delta)\subset O.$
Let $C^n(C)$ denote the cluster of level $n$ containing the set $C$. (It is clear that any connected subset of $K$ must be entirely contained in one cluster.) The clusters $C^n(C), n\in\N$ form a monotone decreasing sequence of sets. Let $L:=\bigcap_n C^n(C)$ be the limit set. Since $C\subseteq C^n(C)$ for every $n$, we have $C\subseteq L$.

We claim that $L$ is connected, which implies at once that $L=C$, since $C$ is a connected component of $K$.
For a proof of the claim, let $y,z\in L$ and $\eps>0$. Then we can find a level $n=n(\eps)$ such that the cylinder sets $\varphi_\omega(K)$ of level $n$ contained in the cluster $C^n(C)$ have diameters less than $\eps$. By definition of the cluster, there are points $y=x_0, x_1,\ldots, x_{p(n)}=z\in C^n(C)$ such that $d(x_i,x_{i+1})<\eps$ for $i=0,\ldots,p(n)-1$. Since $\eps>0$ and $y,z\in L$ were arbitrary, $L$ is connected.

We conclude that, as $n\to\infty$, the sets $C^n(C)$ converge to $C$ in the Hausdorff metric. Hence there exists some level $m$ such that $C^m(C)\subset B(x,3\delta)$. Put $\eps_{0}:=\min\{\eps_{m},\frac{1}{3}\dist(C^{m}(C),\partial O)\}.$
Then, applying Proposition~\ref{prop:w1-2-3-8} to $B:=(C^{m}(C))_{2\eps_{0}}\subset U$, $\beta=1$ and $\eps_{0}$ as above, we obtain that $\tilde s_{0}(K)\geq\dim_M K$,
which contradicts the assumptions.
\end{proof}

\begin{cor} \label{cor:disconn}
Let $K$ be a regular self-similar set in $\er^{d}$ satisfying OSC.
If $K$ is totally disconnected, then $\tilde s_{0}(K)\geq\dim_M K$.
\end{cor}

\begin{rem} \label{rem:cbc}
  Under additional assumptions on the self-similar set $K$ such as polyconvexity of the parallel sets or the curvature bound condition \eqref{eq:cbc} one has $s_k(K)\leq\dim_M K$ and $\tilde s_k(K)\leq\dim_M K$. Thus, under any of those assumptions, one gets the equality $\tilde s_0(K)=\dim_M K$ in  Corollary~\ref{cor:disconn} as well as in Theorem~\ref{thm:conn-compl}.

  Note further that both of these results hold equally with $\tilde s_0(K)$ replaced by $\tilde a_0(K)$, and assuming additionally the integrability condition \eqref{eq:ic} to hold, one has $\tilde a_k(K)\leq\dim_M, K$ which implies again equality in the corresponding statements for $\tilde a_k(K)$.
\end{rem}

\section{Self-similar sets in the plane} \label{sec:plane}

In this section we will prove the following result which characterizes completely degenerate behaviour of the scaling exponents of self-similar sets in $\R^2$ in terms of local flatness, resolving thus Conjecture~\ref{conj} in dimension 2.
\begin{thm} \label{thm:main-R2}
Let $K\subset\er^{2}$ be a self-similar set satisfying OSC.
Then $s_{i}(K)<\dim_M K$ for some $i\in\{0,1\}$ if and only if $K$ is locally flat.
More precisely,
\begin{enumerate}
\item[(i)] $s_1(K)<\dim_M K$ if and only if $K$ is locally $2$-flat (and this happens if and only if $\dim_M K=2$).
\item[(ii)] If $\dim_M K<2$, then ($s_1(K)=\dim_M K$ and) $s_0(K)<\dim_M K$ if and only if $K$ is locally $1$-flat.
\end{enumerate}
\end{thm}

 The statement (i) is a special case of Theorem~\ref{fulldimension}. Moreover, (i) and (ii) together imply the first assertion in Theorem~\ref{thm:main-R2}. Therefore, it remains to provide a proof of (ii), which is the primary aim of the remainder of this section. For this purpose we focus our attention to self-similar sets $K$ with $\dim_M K<2$. The assertion in parentheses is the special case $d=l=2$ of \eqref{eq:s-exponents2} and holds more generally than just for self-similar sets, cf.~also \cite[Corollary~3.6]{rw09}. The structure of the remaining section (and hence of the proof of the remaining assertion in (ii)) is as follows: First we show that for sets in $\R^2$, the scaling exponents coincide with their directed versions, see Corollary~\ref{exp=direxp}. This will enable us in particular to use the results of the previous section.  Then we concentrate on the if part and show in Proposition~\ref{prop:flatness_in R2} (employing Lemma~\ref{distancelemma} and Lemma~\ref{rectangles}) that local $1$-flatness implies $s_0(K)$ to be strictly smaller than the dimension of $K$. The reverse implication is established in a sequence of statements starting from Lemma~\ref{lineseg}. Then assertion (ii) follows by combining Propositions~\ref{prop:flatness_in R2} and \ref{locf2}. We emphasize that for the results in this section, and in particular for Theorem~\ref{thm:main-R2}, no integrability or curvature bound condition is required. However, for sets not satisfying an assumption of this type, a different `degenerate' behaviour is possible, namely $s_0(K)$ can be strictly larger than the dimensions, cf.\ the Cantor dust example discussed in the introduction.

Our first step is to prove that the scaling exponents coincide with their directed versions.
For the cases $k=1$ and $k=2$, this is immediate from the definitions, since the involved measures essentially coincide with their directed versions. For the $0$-th curvature measure and its directed version, we use the following statement. Recall that $\Nor(A,x)$ denotes the normal cone of a set $A\subset\R^2$ at a point $x\in\bd A$.

\begin{lem} \label{curv=dircurv}
   Let $A\subset\R^2$ be a set with positive reach such that $\bd A$ is bounded and such that, for any $x\in A$, $n\in \Nor(A,x)\setminus\{0\}$ implies $-n\notin \Nor(A,x)$.  
   Then, the variation measures of the curvature measures and the curvature-direction measures coincide, i.e., for $\bullet\in\{+,-,\var\}$ and any Borel set $B\subset\R^2$,
   \begin{align*}
     \gC_0^\bullet(A,B\times S^1)=C_0^\bullet(A,B).
   \end{align*}
\end{lem}
Note that for the curvature measures with index $k=1$ and $k=2$ the corresponding statement is trivial.
\begin{proof}
   Recall that for a set $A\subset\R^2$ with positive reach, the generalized principal curvature $\kappa(x,n)$ (in the plane there is only one) is well defined for almost all $(x,n)\in \nor A$.
    We claim that for all $(x,n)\in\nor A$ such that $\kappa(x,n)$ is defined and negative, $n$ is the unique unit normal of $A$ at $x$. Indeed, assuming that $n$ is not the unique unit normal of $A$ at $x$, there is another normal direction $n'\in \Nor(A,x)\cap S^1$ with $n'\neq -n$, and, since $\Nor(A,x)$ is a convex cone, any convex combination of $n$ and $n'$ is also a normal direction. Hence, we have a two-dimensional cone of normal directions. Now note that for inner normal directions in this cone, the generalized principal curvature is $+\infty$, while in the two extremal normal directions, either $\kappa(x,n)$ is not defined or $+\infty$ as well by continuity. Hence, if $\kappa(x,n)$ exists and is negative then $n$ is the unique unit normal of $A$ at $x$ as claimed.

  Let
  $$
  B^-:=\left\{x\in\R^2 |  \Nor(A,x)\cap S^1=\{n\} \text{ for some } n \text{ and } \kappa(x,n)<0\right\},
  $$
  $B^+:=\R^2\setminus B^-$, $\beta^-:=B^-\times S^1$ and $\beta^+:=B^+\times S^1$. Then obviously $\beta^+\cup \beta^-=\R^2\times S^1$ and $\beta^+\cap \beta^-=\emptyset$. Moreover, from the integral representation of $\gC_0(A,\cdot)$ (cf.~\cite[Theorem 3]{Z1}) one gets for each Borel set $\beta\subset\beta^-$,
  \begin{align*}
     \gC_0(A,\beta)=(2\pi)^{-1}\int_{\nor A} \ind{\beta}(x,n) \frac{\kappa(x,n)}
     {\sqrt{1+\kappa^2(x,n)}} \Ha^1(d(x,n))\leq 0,
    \end{align*}
  since the integrand is negative
  for each $(x,n)\in \beta\cap\nor A$ and vanishes otherwise. Similarly, for each Borel set $\beta\subset\beta^+$, we have
  \begin{align*}
     \gC_0(A,\beta)=(2\pi)^{-1}\int_{\nor A} \ind{\beta}(x,n) \frac{\kappa(x,n)}{\sqrt{1+\kappa^2(x,n)}} \Ha^1(d(x,n))\geq 0
  \end{align*}
  since the integrand is nonnegative for each $(x,n)\in \beta^+\cap \nor A$ except for a set of measure zero. Therefore, the pair $(\beta^+, \beta^-$) is a Hahn decomposition of the signed measure $\gC_0(A,\cdot)$. Now recall that $\beta^\pm=B^\pm\times S^1$, which implies immediately that $(B^+, B^-)$ is a Hahn decomposition of the measure $C_0(A,\cdot)$, since $B^+\cup B^-=\R^2$, $B^+\cap B^-=\emptyset$, $C_0(A,B)\leq 0$ for each Borel set $B\subset B^-$ and $C_0(A,B)\geq 0$ for each Borel set $B\subset B^+$. Since $\gC_0^\pm(A,\beta)=\pm\gC_0(A,\beta\cap\beta^\pm)$ and $C_0^\pm(A,B)=\pm C_0(A,B\cap B^\pm)$, we conclude
  \begin{align*}
  \gC_0^\pm(A,B\times S^1)&=\pm \gC_0(A,(B\times S^1)\cap\beta^+)=\pm \gC_0(A,(B\cap B^+)\times S^1)\\&=\pm C_0(A,B\cap B^+)=C_0^\pm(A,B)
  \end{align*}
  for each Borel set $B\subset\R^2$, as asserted in the lemma. The corresponding claim for the total variation measures follows now immediately by adding the positive and the negative variations in the above equation, completing the proof.
\end{proof}

\begin{cor}\label{exp=direxp}
Suppose that $F\subset\R^2$ is bounded. Then for each $k\in\{0,1,2\}$ the scaling exponents $\tilde s_k(F)$ and $s_k(F)$ are well defined and coincide. Similarly, $\tilde a_k(F)=a_k(F)$.
\end{cor}
\begin{proof}
We just prove the case $k=0$. For $F\subset\R^2$, almost all $\eps>0$ are regular values of the distance function of $F$ implying that $F_\eps$ has a Lipschitz boundary and that $\widetilde{(F_\eps)}$ has positive reach. Hence the curvature direction measures $\gC_0(\widetilde{(F_\eps)},\cdot)$ and, -- via the reflection principle -- also $\gC_0(F_\eps,\cdot)$ are well defined for those $\eps$. This is enough to ensure that the scaling exponents $s_0(F)$ and $\tilde s_0(F)$ (given by \eqref{eq:sk-def1} and \eqref{eq:tilde-sk-def1}, respectively) are well defined (possibly $\infty$).

Note that in a Lipschitz boundary no two vectors $n$ and $-n$ can be normals at the same boundary point. Hence Lemma~\ref{curv=dircurv} can be applied and  we have in particular $C_0^\var(F_\eps)=\gC_0^\var(F_\eps)$ for each regular $\eps>0$ from which the equality $\tilde s_0(F)=s_0(F)$ is easily seen.
The relation $\tilde a_0(F)=a_0(F)$ follows similarly from Lemma~\ref{curv=dircurv}.
\end{proof}

Now we will establish that for self-similar sets $K$ local $1$-flatness implies that $s_{0}(K)$ is strictly smaller than $\dim_M K$.
For this, we require the following two lemmas. The first one is a rather simple geometric observation.
We will use the following notation:
For $x\in\er^{2}$ and $a,b>0$ put
$$
R(x,a,b)=x+[-a,a]\times [-b,b]\quad\text{and}\quad Q(x,a,b)=x+[-a,a]\times\{-b,b\}.
$$
\begin{lem}\label{distancelemma}
Let $0<\alpha<1$, $x_{i},$ $i=1,...,N$ be points in $\er^{2}$ and $r_{i},$ $i=1,...,n$ be positive real numbers.
Put
$$
P=\bigcup_{i=1}^{n}R(x_{i},10r_i,r_i)^{\circ},\quad
Q=\bigcup_{i=1}^{n}Q(x_{i},10r_i,r_i),
$$
$$
R=\bigcup_{i=1}^{n}R(x_{i},\alpha r_i,\alpha r_i)\quad \text{and} \quad S=Q\setminus R.
$$
Then for every $x\in P^{c}$ we have
$$
\dist(x,R)>\dist(x,S).
$$
\end{lem}

\begin{proof}
Let $x\in P^{c}$, since $R$ is compact there is some $z\in R$ such that $|x-z|=\dist(x,R).$
Therefore there is some $x_{i}$ such that
$$
\dist(x,R)=\dist(x,R(x_{i},\alpha r_i,\alpha r_i)).
$$

Since we know that $x\not\in R(x_{i},10r_i,r_i)^{\circ}$ there is some $w\in Q(x_{i},10r_i,r_i)$ such that $|x-w|<|x-z|.$
Now, we have either $w\in S$ which would would lead to $\dist(x,S)\leq |x-w|<|x-z|=\dist(x,R)$ which is what we want to prove,
or $w\not\in S$ which means that $w\in R$ and $|x-w|<|x-z|$ which is a contradiction with the choice of $z.$
\end{proof}

\begin{lem}\label{rectangles}
Let $K$ be a self-similar set in $\er^{2}$ satisfying OSC that is not contained in a line segment. Suppose $K$ is locally $1$-flat.
Then there is a system of contracting similarities $\varphi_{1},...,\varphi_{l}$ generating $K$ such that for
the self-similar set $L$, generated by the mappings $\varphi_{1},...,\varphi_{l-1},$
there is a number $0<\alpha<1$ and sequences $\{x_{i}\}_{i=1}^{\infty}\subset\er^{2}$ and $\{r_{i}\}_{i=1}^{\infty}\subset\er_{+}$
such that in some coordinate system in $\er^{2}$ the following properties are satisfied:
\begin{enumerate}
\item $Q(x_{i},10r_i,r_i)\subset K,$ for every $i\in\en$,
\item $K\setminus L\subset\bigcup_{i=1}^{\infty}R(x_{i},\alpha r_i,\alpha r_i),$
\item for every $i\in\en$, there is a nonempty set $A_{i}\subset\er$ such that
$K\cap R(x_{i},10r_i,r_i)$ is a translated and scaled copy of $[-1,1]\times A_{i}.$
\end{enumerate}
\end{lem}
\begin{proof}
We argue similarly as in the proof of Theorem~\ref{fulldimension}.
First we fix some feasible open set $O$ for the SOSC and find a cube $C\subset O$ as guaranteed by the assumption of local $1$-flatness.
Then we choose a coordinate system such that $C\cap K=s[0,1]\times A$ for some $s>0$ and some compact set $A\subset\R$.
Since $K$ is not contained in a line segment, $A$ must be infinite.
Therefore, there is some set of the form $T=Q(x,10r,r)\subset C\cap K$ for some $x\in\er^{2}$ and $r>0$.
Choose $\alpha\in ]0,1[$  such that $R(x,\alpha r,\alpha r)^{\circ}\cap K\not=\emptyset$ and choose $y\in R(x,\alpha r,\alpha r)^{\circ}\cap K.$
Put $D:=\dist(y,R(x,\alpha r,\alpha r)^{c})$.
Let $\psi_{1},...,\psi_{N}$ be similarities generating $K$ with ratios $\rho_{1},...,\rho_{N}.$
If we put $\rho=\max_{i=1,...,N} \rho_{i}>0$, we can find $n\in\en$ such that $\rho^{n}\diam K<D.$

Let $\varphi_{1},...,\varphi_{k}$ be all mappings of the form $\psi_{\omega}$ for $\omega\in\Sigma_k^{n}$ ordered in such a way that $y\in\varphi_{k}(K)$. Let $t_i$ denote the contraction ratio of $\varphi_i$.
Then the IFS $\{\varphi_{1},...,\varphi_{k}\}$ also generates $K$, satisfies OSC (with the feasible open set $O$ from above) and, moreover, $\varphi_{k}(K)\subset R(x,\alpha r,\alpha r)^{\circ}$.
Now let $(x_{i})_{i\in\N}$ be the sequence of all points of the form
$\varphi_{\sigma}(x)$ and $(r_{i})_{i\in\N}$ the sequence of all numbers $rt_{\sigma}$, where $\sigma$ runs through $\Sigma_k^{*}$.

It remains to show that the conditions (1) to (3) hold. For a proof of (1) and (3), fix some $i\in\N$ and let $\sigma$ be the word such that $r_i=rt_\sigma$ and $x_i=\varphi_\sigma(x)$. The inclusion $T\subset K$
implies $Q(\varphi_\sigma(x), 10 rt_\sigma, rt_\sigma)=\varphi_\sigma(T)\subset\varphi_\sigma(K)\subset K$ from which (1) is transparent. Furthermore, since, by construction, $R(x,10r,r)\subset C$, we have
$$
K\cap R(x_i,10r_i, r_i)=K\cap \varphi_\sigma(R(x,10r,r))\subset K\cap\varphi_\sigma(C)=\varphi_\sigma(C\cap K),
$$
where for the last equality we also used that $C\subset O$. Since, by assumption and the choice of the coordinate system, $C\cap K$ is a product set, it is obvious that $\varphi_\sigma(C\cap K)$ is similar to a product set. Observing now that the cube $\varphi_\sigma(C)$ is a valid cube for the condition of local flatness, we can apply Corollary~\ref{flatremark} to infer that line segments in $K\cap\varphi_\sigma(C)$ must be parallel to the first coordinate axis. Hence the same holds for any restriction of this set to a rectangle with sides parallel to the coordinate axes which implies (3).
To prove condition $(2)$, let $x\in K\setminus L$. Since $x\in K$, there is a sequence $\omega=\omega_1\omega_2\omega_3...\in\{1,\ldots,k\}^\N$ such that $\lim_{n\to \infty} \varphi_{\omega|n}(K)=x$ and, since $x\in L^c$, there exists an index $m\in\N$ such that $\omega_m=k$. Hence, putting $\sigma=\omega|m-1$ and letting $i\in\N$ be the index such that $x_i=\varphi_\sigma(x)$ and $r_i=r t_\sigma$, we have
$$
x\in\varphi_\sigma(\varphi_k(K))\subset\varphi_\sigma(R(x,\alpha r,\alpha r))=R(x_i, \alpha r_i, \alpha r_i).
$$
This shows (2).
\end{proof}

Now we have all ingredients to prove one direction of part (ii) of Theorem~\ref{thm:main-R2}.
\begin{prop} \label{prop:flatness_in R2}
Let $K$ be a self-similar set in $\er^{2}$ satisfying OSC with $\dim_M K<2$. If $K$ is locally $1$-flat then $s_{0}(K)<\dim_M K$.
\end{prop}
\begin{proof}
If $K$ is contained in a line segment, we are done due to Proposition~\ref{embedding fulldimension}.
In the other cases let $L$, $\{x_{i}\}_{i=1}^{\infty}\subset\er^{2}$ and $\{r_{i}\}_{i=1}^{\infty}\subset\er_{+}$ be as guaranteed by Lemma~\ref{rectangles}.
In the sequel, we also work in the coordinate system we get from this statement.
Note that $\dim_M L<\dim_M K$, since $\sum_{i=1}^{k-1} r_{i}^{\dim_M K}<1$.
Let $M:=K\setminus L$,
$$
P_{j}:=\bigcup_{i=1}^{j}R(x_{i},8r_i,r_i)
\quad
\text{and}
\quad
M_{j}:=M\cap\bigcup_{i=1}^{j}R(x_{i},\alpha r_i,\alpha r_i).
$$
Then we have
$$
M=\bigcup_{i=1}^{\infty}M_{j},
$$
and, by (2) of Lemma~\ref{rectangles}, $K\setminus M_{j}$ converges to $L$ in the Hausdorff metric as $j\to\infty$.

Fix some regular $\eps>0$ such that $C_0(K_\eps,\cdot)$ and $C_0(L_\eps,\cdot)$ are well defined. Let $R_i:=R(x_{i},8r_i,r_i)$. First we claim that, for every $i\in\en$,
\begin{align} \label{eq:supp-Cvar}
  \supp C^{\var}_{0}(K_{\eps},\cdot)\cap R_{i}^{\circ}=\emptyset,
\end{align}
which implies in particular that
$
C^{\var}_0(K_{\eps}, R_{i}^{\circ})=0.
$
For a proof of \eqref{eq:supp-Cvar}, let $y\in \bd K_\eps \cap R_{i}^{\circ}$. (Note that $\supp C^{\var}_{0}(K_{\eps},\cdot)\subset \bd K_\eps$.) It suffices to show that $y\not\in\supp C^{\var}_{0}(K_{\eps},\cdot)$.
The assumption implies that $r_i>\eps$, since otherwise $\bd K_\eps\cap R_i=\emptyset$.
 Taking into account property (3) of Lemma~\ref{rectangles}, we infer that $y=(y_1,y_2)$ has a unique nearest point $x$ in $K$. Moreover, $x$ lies in $R_i$ and has coordinates $x=(y_1,y_2+\eps)$ or $x=(y_1,y_2-\eps)$. Moreover, $\dist(y,K\setminus R_i)>\eps$ and so there is some $r>0$ such that $X:=\bd K_\eps \cap B(y,r)$ is a line segment (parallel to the $x_1$-axis) which obviously does not carry any curvature (formally, by locality, we have $C_0^\var(K_\eps, B^\circ(y,r))=C_0^\var(X+[0,x-y],B^\circ(y,r))=0$). This implies $y\notin\supp C^{\var}_{0}(K_{\eps},\cdot)$ proving the claim.

Put $T:=\bd K_{\eps}\cap (\bigcup_{i=1}^\infty R_{i}^{\circ})^{c}.$ Note that $T$ is compact. Moreover, from \eqref{eq:supp-Cvar}, we observe that
$$
 C^{\var}_{0}(K_{\eps},\bigcup_{i=1}^\infty R_{i}^{\circ})\leq \sum_{i=1}^\infty C^{\var}_{0}(K_{\eps},R_{i}^{\circ})=0
$$
and thus $C^{\var}_{0}(K_{\eps})=C^{\var}_{0}(K_{\eps},T)$.
From Lemma~\ref{distancelemma} and property $(3)$ of Lemma~\ref{rectangles} we infer that, for every $y\in T$ and every $j\in\N$,
$$
\dist(y,M_{j})>\dist(y,P_{j}\cap K)\geq\dist(y,K)=\eps,
$$
and using the fact that $T$ is compact, we obtain $\dist(T,M_{j})>\eps.$
This means that there is an open set $U\supset T$ such that $K_{\eps}\cap U=(K\setminus M_{j})_{\eps}\cap U.$
By the locality property of the curvature measures, we conclude
$$
C^{\var}_{0}(K_{\eps})=C^{\var}_{0}(K_{\eps},T) =C^{\var}_{0}((K\setminus M_{j})_{\eps},T).
$$
Using \cite[Theorem~5.2]{RSM} and the fact that $(K\setminus M_{j})$ converges to $L$ in the Hausdorff metric as $j\to\infty$, we obtain
$$
C^{\var}_{0}(K_{\eps})= C^{\var}_{0}(L_{\eps},T)\leq C^{\var}_{0}(L_{\eps}).
$$
Therefore
$$
s_{0}(K)\leq s_{0}(L)\leq\dim_M L<\dim_M K.
$$
\end{proof}

To complete the proof of Theorem~\ref{thm:main-R2} (ii),
it remains to show that the inequality $s_{0}(K)<\dim_M K$ implies that $K$ is locally $1$-flat, provided that $\dim_M K<2$. The argument is split into several pieces.

\begin{lem}\label{lineseg}
Let $K\subset\er^{2}$ be a self-similar set satisfying OSC such that $s_{0}(K)<\dim_M K<2.$
Suppose that $O$ is a feasible open set and that $C$ is a connected component of $O\cap K$ containing at least two points.
Then $C$ is a line segment.
\end{lem}

\begin{proof}
First recall that every connected component of $O\cap K$ is also path connected (see e.g.~\cite[Theorem 1.6.2]{K}).
Suppose for a contradiction that $C$ is as in the statement of Lemma~\ref{lineseg} but not a line segment.
Then there is a simple curve $\gamma\subset C\cap O$ with endpoints $x$ and $y$, and a third point $z$ on $\gamma$ with $z\notin[x,y]$ such that $B(z,\delta)\cap\gamma$ is not a line segment for any $\delta>0$. Let $E$ be the open set enclosed by $\gamma\cup [x,y]$. We can assume that there is some $\eps>0$ such that $\overline{E}\subset O_{-\eps}$. (If this is not satisfied, there is some $\eps'>0$ such that $B(z,2\eps')\subset O$. Since $B(z,\eps')\cap \gamma$ is not a line segment, one can choose new points $x'$ and $y'$ on $\gamma$ such that the subcurve $\gamma'$ of $\gamma$ from $x'$ to $y'$ contains $z$ and is contained in $B(z,\eps')\subset O_{-\eps}$. One can use $\gamma'$ as the new $\gamma$, which obviously satisfies the desired inclusion $\overline{E}\subset O_{-\eps}$ for $\eps=\eps'$.)
Let $L$ denote the line through $x$ and $y$ and set $D:=\frac{1}{2}\dist(z,L)>0$.
Since $\dim_M K<2$, the interior of $K$ is empty and we can choose $\eps_0\in(0,\eps)$ such that for every $0<\alpha<\eps_{0}$ there is some point $u_{\alpha}\in \bd K_{\alpha}\cap B^o(z,D)\cap E$. ($\eps_0$ and points $u_\alpha$ can be found as follows: Let $u$ be a point in $K^c\cap B^o(z,D)\cap E$, which exists since $B^o(z,D)\cap E$ has positive area, while $K\cap B^o(z,D)\cap E$ is a null set. Let $U$ be the connected component of $K^c\cap E$ containing $u$. $U$ is obviously open. Now $\eps_0$ can be chosen to be the inradius of $U$ and for $0<\alpha<\eps_0$ one can take any point of $\bd U_{\alpha}\cap U$ as $u_\alpha$.)
Let $A^{\alpha}$ be the connected component of $\partial K_{\alpha}$ containing $u_{\alpha}$.
We want to estimate $C^{\var}_{0}(K_{\alpha},A^{\alpha}\cap O_{-\eps})$ from below independently of $\alpha$. It is enough to consider $\alpha$ that are regular for $K$.

 Note that, as part of the boundary of a parallel set, $A^{\alpha}$ is a (not necessarily simple but closed) $C^1$ Jordan curve. There are two cases to consider. Either there is a simple curve $\Gamma^{\alpha}\subset A^{\alpha}$ containing $u_{\alpha}$ with endpoints $a_{\alpha},b_{\alpha}\in[x,y]$, or there is a loop $L^\alpha\subset A^\alpha\cap E$ with $u_{\alpha}\in L^{\alpha}$.
In the second case, we have $C^{\var}_{0}(K_{\alpha},A^{\alpha}\cap O_{-\eps})\geq C^{\var}_{0}(K_{\alpha},L^{\alpha})\geq 1$, by Lemma~\ref{lem:C0var}.

In the first case, the set enclosed by the loop
$\Gamma_\alpha \cup [a_\alpha,b_\alpha]$ has positive reach, since $\alpha$ was assumed to be  regular for $K$, and we have the following:
For every unit vector $v$ between $v_a$ and $v_b$, the (outward) unit normals of the
segments $[a_\alpha,u_\alpha]$ and $[u_\alpha,b_\alpha]$ in the triangle $a_\alpha u_\alpha b_\alpha$, there must be a point $p=p(v)$ in the
relative interior of $\Gamma_\alpha$ such that  $v\in \Nor(\Gamma_\alpha,
p)$. Using again Lemma~\ref{lem:C0var}, we get
$
C^{\var}_{0}(K_{\alpha},\Gamma^{\alpha}\cap O_{-\eps})\geq\beta_{\alpha},
$
where $\beta_{\alpha}$ is the angle between $v_a$ and $v_b$ (or equally, the exterior angle at $u_\alpha$ of the triangle $a_\alpha u_\alpha b_\alpha$). Observe that $\beta_\alpha$ is bounded from below by some positive constant $\beta$ independent of $\alpha$: Since $a_\alpha$ and $b_\alpha$ are points on the segment $[x,y]$, $\beta_\alpha$ is certainly larger than the corresponding angle in the triangle $xu_\alpha y$. Since $\alpha<\eps_0<D$, this angle cannot be smaller than the corresponding angle $\beta$ of an equilateral triangle with base $[x,y]$ and height $D$ (which minimizes this angle among the triangles with base $[x,y]$ with height $D$).
We obtain
\begin{equation*}
C^{\var}_{0}(K_{\alpha},A^{\alpha}\cap O_{-\eps})\geq C^{\var}_{0}(K_{\alpha},\Gamma^{\alpha}\cap O_{-\eps})\geq\beta_{\alpha}\geq\beta:=\frac{D}{\sqrt{|x-y|^{2}+4D^{2}}}>0.
\end{equation*}
Applying Theorem~\ref{prop:w1-2-3-8} to $K$ with $B=E\subset O_{-\eps}$ and $\beta$ and $\eps_{0}$ as above,
we obtain that $s_{0}(K)=\dim_M K$ which contradicts assumptions of the lemma.
\end{proof}

\begin{cor}\label{compline}
Let $K\subset\er^{2}$ be a self-similar set satisfying OSC such that $s_{0}(K)<\dim_M K<2.$
Then any connected component of $K$ is a (possibly degenerated) line segment.
\end{cor}
\begin{proof}
Let $C$ be a connected component of $K$.
We can suppose that $C$ contains at least two points since otherwise the assertion is obvious.
Choose $O$ to be some feasible open set for the SOSC and let $x\in K\cap O$.
Let $\eps=\dist(x,\partial O)$.
Let $\varphi_{1},...,\varphi_{k}$ be similarities generating $K$ with ratios $r_{1},...,r_{k}.$
If we put $r:=\min_{i=1,...,k} r_{i}>0$, we can find $n\in\en$ such that $r^{n}\diam K<\eps.$
Choose $\omega\in\Sigma_k^{n}$ such that $x\in\varphi_{\omega}(K)$ and let $C'=\varphi_\omega(C)$. By construction, we have $C'\subset\varphi_\omega (K)\subset O\cap K$.
Let $A$ be the connected component of $K\cap O$ such that $C'\subset A.$
Then, by Lemma~\ref{lineseg}, $A$ is a line segment and therefore the same is true for $C'$ and $C.$
\end{proof}

As a byproduct we obtain the following result which complements the results in Theorem~\ref{thm:conn-compl} and Corollary~\ref{cor:disconn} in providing another simple geometric condition which ensures the $0$-th scaling exponent to coincide with the dimension. Again, under the curvature bound condition \eqref{eq:cbc} or polyconvexity, the inequality in Corollary~\ref{cor:connected_sets} becomes an equality, cf.~Remark~\ref{rem:cbc}.

\begin{cor} \label{cor:connected_sets}
Suppose that $K$ is a connected self-similar set in $\er^{2}$ satisfying OSC with $1<\dim_M K<2.$
Then $s_{0}(K)\geq\dim_M K$.
\end{cor}
\begin{proof}
   Suppose for a contradiction that $s_0(K)<\dim_M K$. Then, by Corollary~\ref{compline} and the connectedness of $K$, $K$ must be a line segment. But this implies $\dim_M K=1$ in contradiction with the assumption $1<\dim_M K$.
\end{proof}

In Corollary~\ref{compline}, it is stated that all connected components of $K$ are line segments. The next step towards to proof of local $1$-flatness is to show that all the segments in $K$ are parallel to each other. This follows immediately from the following statement.

\begin{lem}
Let $K\subset\er^{2}$ be a self-similar set satisfying OSC such that $s_{0}(K)<\dim_M K<2$.
Let $\varphi_{1},...,\varphi_{k}$ be similarities generating $K.$
Then no $\varphi_{i}$ can include a rotation by an angle different to $\pi.$
\end{lem}

\begin{proof}
Suppose for a contradiction that one of the mappings, say  $\varphi_{b}$, includes a rotation by angle $\alpha\notin\{0,\pi\}$.
Combining Corollary~\ref{compline} and Lemma~\ref{bigpieces} there exists some $z\in K$ and $\delta>0$ such that
every connected component of $K$ that intersects $B(z,\delta)$ is a line segment with length bigger than $\delta.$
Let $\Gamma$ be the system of all connected components of $K$ intersecting $B(z,\delta).$
Then $\Gamma$ is infinite since otherwise $K\cap B(z,\gamma)=A\cap B(z,\gamma)$, for some $\gamma>0,$ where $A$ is the connected component of $K$ containing $z.$
But this would mean that $K$ contains a line segment and is contained in a line segment which is not possible due to the existence of $\varphi_{b}$.
Moreover $\Gamma$ is closed in the Hausdorff metric.
Therefore $\Gamma$ has an accumulation point, which means in particular that there are pairwise disjoint line segments $L_{i}=[x_{i},y_{i}],L=[x,y]\subset K,$ with $|x_{i}-y_{i}|=|x-y|=\delta$ and such that $L_i\to L$ in the Hausdorff metric.

Put $c_{i}=\frac{x_{i}+y_{i}}{2}.$
Without loss of generality we can suppose that $x=(-\frac{\delta}{2},0)$ and $y=(\frac{\delta}{2},0).$
Then $c_i\to(0,0)$ as $i\to\infty.$
Choose $p,q$ in a way that $|c_{p}|,|c_{q}|<\frac{\delta}{30}.$
We can suppose that $c_{q}\in \conv(L\cup L_{p}).$
Define $\rho:=\dist(c_{q},L\cup L_{p})$, let $r_{i}$ be the similarity ratio of $\varphi_{i}$ and $r:=\max_{i}r_{i}.$
Choose $n\in\en$ such that $r^{n}\diam K<\rho$ and find $\omega\in\Sigma_k^{*}$ in a way that $|\omega|>n$ and such that $c_q\in \varphi_\omega(K)$.
In particular, this means that $\varphi_{\omega}(L)\subset\conv(L\cup L_{p}).$
Now, on the one hand, for some $k>0$, the direction of the line segment $\varphi_{\omega}\circ\varphi_b^k(L)$ is contained in $[\frac{\pi}{4},\frac{3\pi}{4}]$. Let $M$ be the connected component of $K$ containing $\varphi_{\omega}(L)\circ\varphi_b^k(L)$. By Corollary~\ref{compline}, $M$ is a line segment (with direction in $[\frac{\pi}{4},\frac{3\pi}{4}]$), and, by Lemma~\ref{bigpieces}, the length of $M$ is bigger than $\delta$ (since $\varphi_{\omega}(L)\circ\varphi_b^k(L)\cap B(z,\delta)\not=\emptyset$).
On the other hand, $\varphi_{\omega}\circ\varphi_b^k(L)\subset\conv(L\cup L_{p})$.
But this is not possible since in this case $M$ would intersect $L$ or $L_{p}.$
\end{proof}

\begin{cor}\label{parallel}
Let $K\subset\er^{2}$ be a self-similar set satisfying OSC such that $s_{0}(K)<\dim_M K<2.$
Then all components of $K$ are mutually parallel line segments or singletons.
\end{cor}

The following statement provides the last missing piece of the proof of Theorem~\ref{thm:main-R2}.

\begin{prop}\label{locf2}
Let $K\subset\er^{2}$ be a self-similar set satisfying OSC such that $s_{0}(K)<\dim_M K<2.$
Then $K$ is locally $1$-flat.
\end{prop}

\begin{proof}
Fix some feasible open set $O$ for $K$ and find $U\subset O$ and $\delta>0$ as in Lemma~\ref{bigpieces}.
By Corollary~\ref{parallel}, all components of $K$ intersecting $U$ are mutually parallel line segments.
Let $v$ be a unit vector parallel to the direction of the components of $K$ and $v^\perp$ be a unit vector orthogonal to $v$.
Choose some $x\in K\cap U$ and some $0<\eps<\delta$ such that
$$
C:=[x,x+\eps v^{\bot}]\times[x,x+\eps v]\subset U_{-\eps}\subset O_{-\eps}.
$$
Without loss of generality, we can suppose $\eps=1$, $v=(0,1)$ and $x=(0,0)$.
Then $C=[0,1]\times[0,1].$
Define $P_0:=K\cap([0,1]\times\{0\})$, $P_1:=K\cap([0,1]\times\{1\})$ and let $P\subset\R$ be the projection of $P_0\cup P_1$ onto the first coordinate.
Note that $P$ is compact and totally disconnected since both $P_0$ and $P_1$ are (which is due to Lemma~\ref{parallel}).
We are done if we prove that $C\cap K=P\times[0,1]$, since in this case the local $1$-flatness of $K$ follows from Proposition~\ref{prop:one-point-flat}.

Suppose that this is not true.
Then there is some $p\in P$ such that the segment $S_p:=\{p\}\times [0,1]$ is not completely contained in $K$. Since, by construction, $S_p\cap K$ is closed and nonempty, there are $0\leq a<b\leq 1$ such that
$$
S_p\cap K\subset \{p\}\times ([0,a]\cup[b,1]).
$$
Fix some $0<\tau<\frac{b-a}{3}$. Since $K$ is closed, there is some $\gamma>0$ (with $\gamma<\tau$) such that
$$
S_q\cap K\subset \{q\}\times ([0,a+\tau]\cup[b-\tau,1])
$$
for every $q\in [p-\gamma,p+\gamma]\cap [0,1]$.

Since $P$ is compact and totally disconnected, we can find $c,d\in [p-\gamma,p+\gamma]\cap [0,1]\cap P$ such that the open interval $(c,d)$ is disjoint from $P$.
Because $c\in P$ and by Lemma~\ref{bigpieces}, the set $(\{c\}\times[0,1])\cap K$ is of the form $\{c\}\times ([0,a_c]\cup[b_c,1])$ (or with possibly one of the intervals $[0,a_c]$ or $[b_c,1]$ missing).
Since $c\in [p-\gamma,p+\gamma]\cap [0,1]$, we also have $b_c-a_c>\tau.$
Now, for every $r\leq\eps_0:=\min\{|c-d|/3,\tau/3\}$, the set $\partial K_r$ contains a quarter of a circle centered either at $(c,a_c)$ or $(c,b_c)$.
To complete the proof, observe that Proposition~\ref{prop:w1-2-3-8} (applied to $B=C$, $\beta=\frac{1}{4}$ and $\eps_{0}$ as just defined) together with Corollary~\ref{exp=direxp} implies $s_0(K)\geq\dim_M K$, a contradiction to the assumptions of Proposition~\ref{locf2}.
\end{proof}

\section{Final Remarks} \label{sec:final}

Summing up, we have shown that, although generically all the curvature scaling exponents of a self-similar set coincide, there are nontrivial sets which do not show such generic behaviour. We have demonstrated that nongeneric behaviour is closely connected with the notion of local flatness -- it is characteristic at least in dimensions 1 and 2.

\medskip

\paragraph{\bf Possible combinations of scaling exponents.} In the introduction we have raised the question, for which vectors $(t_0,\ldots, t_d)\in\R^{d+1}$ there exists a self-similar set $K\subset\R^d$ with $s_k(K)=t_k$ for $k=0,\ldots,d$, which we will briefly address now. First of all it should be noted that $0\leq\udim_M F\leq d$ for any bounded set $F\subset\R^d$, implying that only for vectors with $0\leq t_d\leq d$ such a set can exist. The same constraints apply to $t_{d-1}$, as is transparent from the results in \cite{rw09}. In fact, we get a much stronger constraint from the fact (proved in \cite{rw09}) that for any bounded set $F\subset\R^d$ either $s_{d-1}(F)=s_d(F)$ or $\lambda_d(F)>0$ (implying $s_d(F)=d$ and $s_{d-1}(F)\geq {d-1}$). Effectively, this reduces the problem by one dimension leaving (subsets of) two hyperplanes of possible parameter vectors.  Since, by definition, $s_k\geq 0$ for all $k$, we have also the  constraint $t_k\geq 0$.  Imposing additionally the curvature bound condition \eqref{eq:cbc}, leads to the constraint $t_k\leq t_d$ for the parameter vectors, cf.~\cite[Theorem~2.2]{WZ}.  In $\R^2$, for instance, all vectors of scaling exponents must either be of the form $(t_0, D, D)$ with $0\leq t_0\leq D$ and $D\in[0,2]$ or of the form $(t_0,t_1,2)$  with $0\leq t_0\leq 2$ and $t_1\in[1,2]$. Figure~\ref{fig:sec7} illustrates, for self-similar sets in $\R^2$, which combinations of scaling exponents may be possible and for which combinations of scaling exponents some self-similar sets are known.

\begin{figure}
\begin{minipage}{65mm}
  \includegraphics[width=67mm]{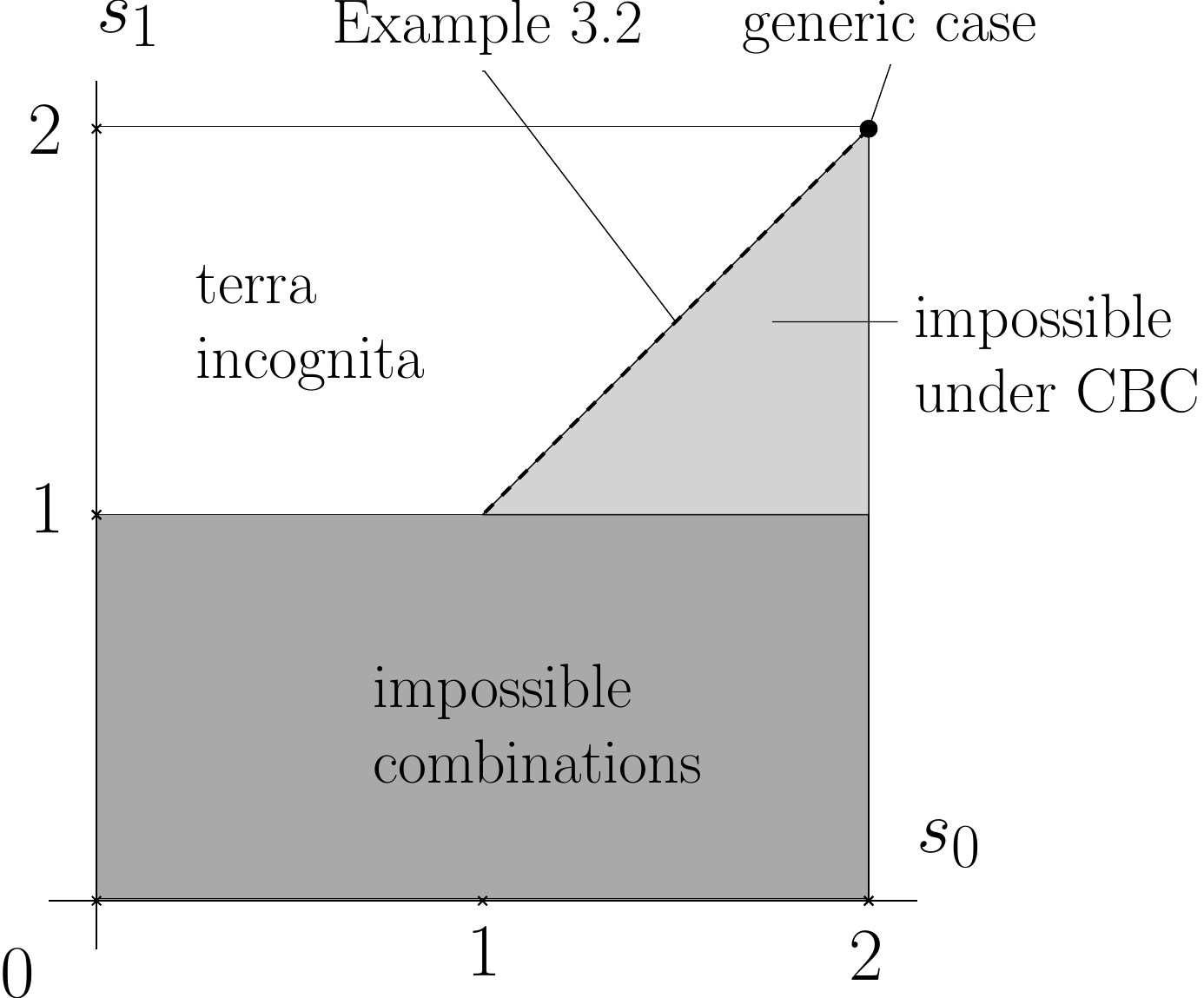}
 \end{minipage}
\begin{minipage}{60mm}
  \includegraphics[width=53mm]{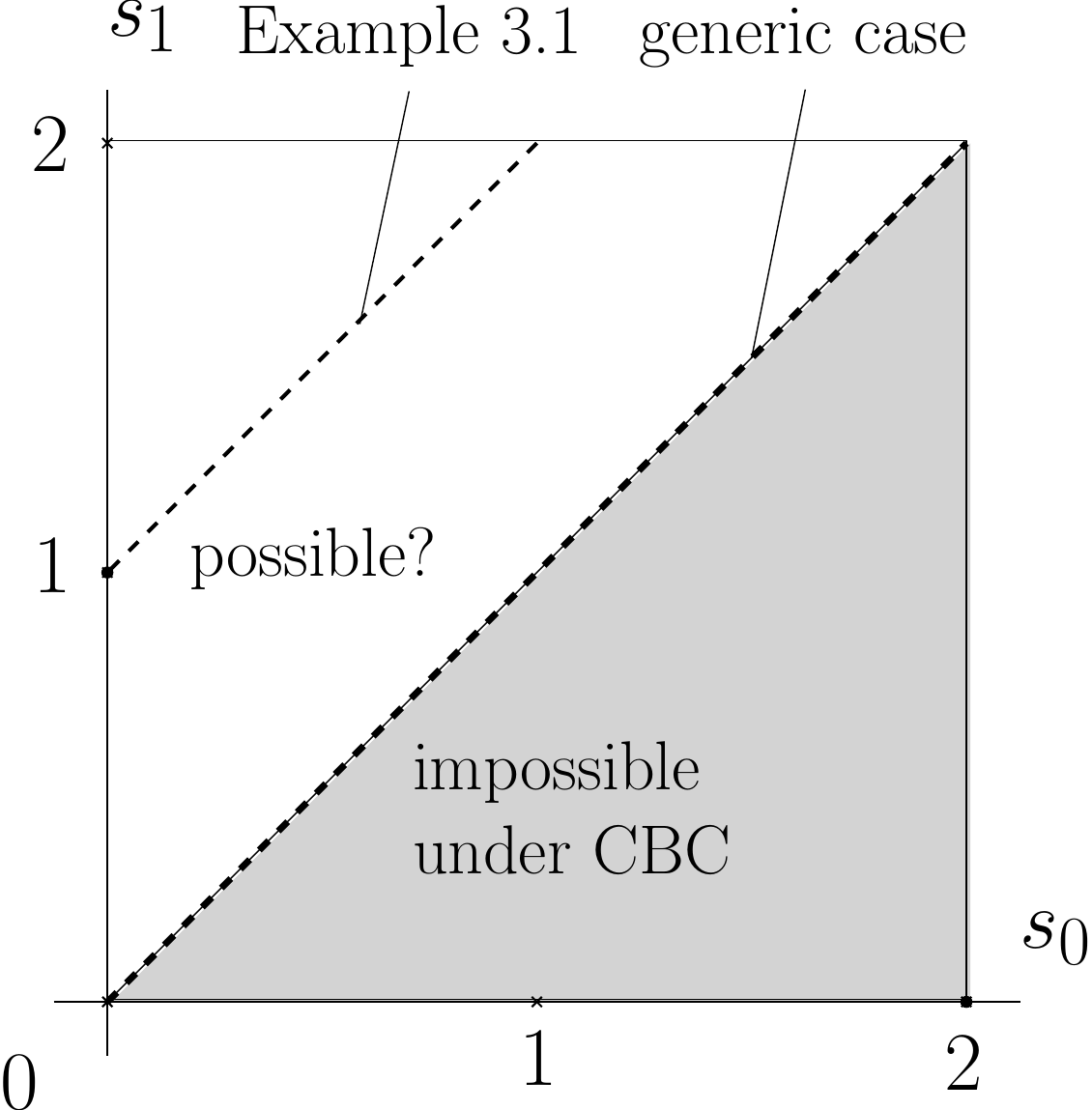}
\end{minipage}
\caption{\label{fig:sec7} Possible and impossible combinations of scaling exponents $s_0,s_1$ for self-similar sets in $\R^2$ in the case $s_2=2$ (Left) and in the case $s_1=s_2$ (Right). Grey shaded regions refer to combinations confirmed to be impossible, white regions to combinations that are still possible. The dotted lines refer to pairs of scaling exponents for which sets are known.}
\end{figure}
Note that for the sets in Example~\ref{ex:1}, the scaling vectors are of the form $(t_0,t_0+1,t_0+1)$ with $t_0\in[0,1]$ (that is they are all on one line in the above mentioned triangle), while the sets in Example~\ref{ex:2} are of the form $(t_0,t_0,2)$ with $t_0\in[1,2]$ (that is, they are all on one line in the rectangle).

If the same questions are asked for the average scaling exponents $a_k(K)$, the answers are exactly the same, if the (weaker) integrability condition \eqref{eq:ic} is imposed instead of CBC \eqref{eq:cbc}.

It is an open question, for which vectors within the spotted regions there exist self-similar sets with those scaling exponents. It would in particular be good to know, whether the relation $s_0\le s_1 \le \ldots \leq s_{d-1}$ holds in general. More specifically, is it true that all locally $1$-flat self-similar sets in $\R^2$ satisfy the equation $s_1(K)=s_0(K)+1$ that we found for the sets in Example~\ref{ex:1}? We hope that further investigations will provide answers to these questions.

\medskip

\paragraph{\bf Scaling exponents of general sets.} The following example shows that for general sets in the plane we can expect a much wider behaviour than in the self-similar setting. We prescribe the scaling exponents $s_0$ and $s_1$ (within a certain range) and construct a set with exactly these exponents. This is much more than we were able to do in the self-similar setting; in Example~\ref{ex:1} we only constructed sets with $s_1-s_0=1$.

\begin{ex} \label{ex:general}
Suppose that $2> b\geq a\geq b-1>0$.
Then there is a compact set $K=K(a,b)\subset\er^{2}$ such that $s_{0}(K)=a$ and $s_{1}(K)=b(=s_2(K)).$
\end{ex}
\begin{proof}
Fix $a,b$ as above.
Then there is a number $q$ with $0<q<\frac{1}{2}$ such that $2q^{b-1}=1.$
For $n\in\N_0$ and each $i\in\N$ such that $2^{n}\leq i<2^{n+1},$ define $r_{i}:=q^{n}(1-2q)$ and $t_{i}:=\left\lfloor \frac{q^{(b-a-1)n}}{1-2q}\right\rfloor$, where $\lfloor x\rfloor$ denotes the integer part of a number $x\geq 0$.
Then $r_{i}\geq r_{i+1}$, $\sum_{i=1}^{\infty} r_{i}=1$ and $1\leq t_{i}\leq \frac{1}{r_{i}}$.
Set $p_{i}:=\sum_{k=1}^{i}r_{k}$ and define the set $K=K(a,b)\subset\R^2$ by
\begin{equation*}
K=\bigcup_{i\in\en}\left(\left(\{p_{i-1}\}\times[0,1]\right)\cup \bigcup_{l=0}^{t_{i}}\left([p_{i-1},p_{i}]\times\left\{\frac{l}{t_{i}}\right\}\right)\right)\cup \left(\{1\}\times[0,1]\right).
\end{equation*}
Geometrically, the set $K$ is obtained by dividing the unit square into rectangles by vertical line segments with distances $r_{i}$ (to control $s_{1}(K)$)
and then dividing each of these rectangles into $t_{i}$ similar rectangles by adding horizontal line segments (to control $s_{0}(K)$).

Let $\varepsilon_{n}=\frac{1}{2}q^{n}(1-2q)=\frac{r_i}2$. Then, for  $\varepsilon_{n+1}\leq\varepsilon<\varepsilon_{n}$,
\begin{equation}\label{s0}
C^{\var}_{0}(K_{\varepsilon})=1+ \sum_{k=1}^{2^{n+1}-1}t_{k}
\end{equation}
and, since $C_1^\var(K_\eps)=C_1(K_\eps)$ is half the boundary length of $K_\eps,$
\begin{equation}\label{s1}
C^{\var}_{1}(K_{\varepsilon})=2+\pi \eps+\sum_{k=1}^{2^{n+1}-1}(1+t_{k}r_{k})-4\varepsilon\sum_{k=1}^{2^{n+1}-1}t_{k}.
\end{equation}
Due to the definition of $t_{i}$ and by the summation formula for geometric series, there are constants $c_1,c_2>0$ and $c_3$, where $c_3=0$ when $a=b,$ such that
\begin{equation*}
c_1(2q^{b-a-1})^{n}\geq 1+\sum_{k=1}^{2^{n+1}-1}t_{i}\geq c_2(2q^{b-a-1})^{n}-c_3 2^{n}.
\end{equation*}
Combining this with (\ref{s0}) and multiplying $\eps^a$, we infer that there are constants $c_1',c_2'>0$ and $c_3'$, where $c_3'=0$ when $a=b$, such that
\begin{equation*}
c_1'q^{an}(2q^{b-a-1})^{n}\geq \varepsilon^{a}C^{\var}_{0}(K_{\varepsilon})\geq c_2'q^{an}(2q^{b-a-1})^{n}-c_3'(2q^{a})^{n}.
\end{equation*}
Since $2q^{b-1}=1$ and either $0<2q^{a}<1$ or $c_3'=0$, we conclude that, for $n$ sufficiently large,
\begin{equation}\label{ests1}
c_1'\geq \varepsilon^{a}C^{\var}_{0}(K_{\varepsilon})\geq \frac{c_2'}{2}
\end{equation}
and therefore $s_{0}(K)=a.$
Similarly, using (\ref{s0}) and (\ref{s1}), there are constants $c_4,c_5>0$ such that
\begin{equation*}
c_4 2^{n}\geq C^{\var}_{1}(K_{\varepsilon})\geq c_5 2^{n} - 4\varepsilon C^{\var}_{0}(K_{\varepsilon}).
\end{equation*}
Using (\ref{ests1}), we infer that there are $c_4',c_5',c_6>0$ such that
\begin{equation*}
c_4'q^{(b-1)n}2^{n}\geq \varepsilon^{b-1}C^{\var}_{1}(K_{\varepsilon})\geq c_5'q^{(b-1)n}2^{n}-c_6\varepsilon^{b-a}.
\end{equation*}
Since $2q^{b-1}=1$ and $b-a>0$, we conclude that, for $n$ sufficiently large,
\begin{equation*}
c_4'\geq\varepsilon^{b-1}C^{\var}_{1}(K_{\varepsilon})\geq \frac{c_5'}{2}
\end{equation*}
and therefore $s_{1}(K)=b.$
\end{proof}

It is not difficult to see that in the example, the exponents $s_k$ can be replaced by $a_k$. It  remains an interesting open question, whether there exist sets $F\subset\R^2$ such that $a_0(F)>a_1(F)$. We believe this is not possible, however, up to now we have not been able to prove this.

\medskip

\paragraph{\bf Compatible self-similar tilings.}
Given a self-similar IFS $\{\varphi_1,\ldots,\varphi_N\}$ in $\R^d$ satisfying the OSC and a feasible open set $O$, in \cite{PW} a tiling $\sT=\sT(O)$ of the set $O$ is defined by setting $G:=O\setminus \Phi(\overline{O})$ (where $\Phi(A):=\bigcup_{j=1}^N \varphi_j(A)$ for $A\subset\R^d$) and
$$
\sT:=\{\varphi_\sigma(G): \sigma\in\Sigma^*_N\},
$$
that is, the tiles of $\sT$ are the iterates of the (open) set $G$, which is called the \emph{generator} of $\sT$.  Whenever the set $G$ is nonempty (which happens if and only if the associated self-similar set $F$ has no interior points, i.e., if $\dim_M F=d$), the family $\sT$ is a tiling of $O$ in the sense that the tiles $\varphi_\sigma(G)$ of $\sT$ are pairwise disjoint and the closure of their union equals the closure of $O$, i.e.\
$$
\overline{O}=\overline{\bigcup_{R\in\sT} R},
$$
see \cite[Theorem 5.7]{PW}.
Let $F$ be the self-similar set associated to $\{\varphi_1,\ldots,\varphi_N\}$. The self-similar tiling $\sT$ is called \emph{compatible}, if and only if $\bd O \subset F$. Compatibility is equivalently characterized by the condition $\bd G\subset F$ or by the equation
\begin{align}  \label{eq:compatible}
(F_\eps\setminus F)= T_{-\eps}\cup ({O}_\eps\setminus O)
\end{align}
for any (and thus all) $\eps>0$, see \cite[Theorem 6.2]{PW}.
Self-similar tilings have been used as a tool to study the geometric properties of self-similar sets, in particular,
to obtain fractal tube formulas and to introduce complex dimensions for self-similar sets in $\R^d$, see e.g.~\cite{LapvF,LP,LPW1}. These results have for instance been used in the characterization  of Minkowski measurability, see e.g.~\cite{LPW2}.
In view of equation \eqref{eq:compatible}, compatibility allows to transfer results from  tilings to the associated sets, and hence to replace the study of self-similar sets by the study of self-similar tilings, which turned out to be much easier in certain cases.

It is therefore an interesting question, to characterize those self-similar sets which possess a compatible self-similar tiling. That is, given a self-similar set $F$ (satisfying OSC and $\dim_M F<d$), does there exist a feasible set $O$ such that $\sT(O)$ is a compatible self-similar tiling? It is known from \cite{PW} that there exist self-similar sets (e.g.\ the Koch curve) which do not possess a compatible tiling. In fact, it is not difficult to see that a self-similar set $F$ possesses no compatible tiling if the complement of the set $F$ is connected, see \cite[Proposition~6.3]{PW}. Using an argument from the proof of Theorem~\ref{thm:conn-compl}, we can strengthen this observation to an if-and-only-if statement. A self-similar set $F$ has a compatible tiling if and only if its complement is not connected.

\begin{thm} \label{thm:tiling}
  Let $F$ be a self-similar set in $\R^d$ satisfying OSC and $\dim_M F<d$. Then the set $F$ possesses a compatible self-similar tiling $\sT$ (of some suitable feasible set $O$) if and only if $F^c$ is disconnected.
\end{thm}
\begin{proof}
If $F$ has a compatible tiling $\sT$ (of some feasible set $O$), then its generator $G$ satisfies $\bd G\subset F$. Since $F$ cannot cover the whole open set $G$, there must be a connected component of $F^c$ contained in $G$ which is bounded and thus not the unbounded connected component of $F^c$. Hence $F^c$ is disconnected, proving one direction.

For the reverse implication, assume that $F^c$ is disconnected or, which is the same, that $F^c$ has a bounded connected component $B\subset F^c$. Let $\{\varphi_1,\ldots,\varphi_N\}$ be an IFS generating $F$ and let $O$ be an arbitrary strong feasible open set for $F$. By the first part of the proof of Theorem~\ref{thm:conn-compl}, we can assume without loss of generality that $B\subset O$. Using $B$ we construct a new feasible open set $U$ for $F$ by setting
 $$
 U:=\bigcup_{\sigma\in\Sigma^*_N} \varphi_\omega(B).
 $$
 Indeed, it is easily seen that $\varphi_i(U)\subset U$ for $i=1,\ldots,N$. Moreover, since $B\subset O$ and thus $\varphi_\sigma(B)\subset \varphi_\sigma(O)\subset O$ for any $\sigma\in\Sigma^*_N$, we have $U\subset O$
 from which $\varphi_i(U)\cap \varphi_j(U)=\emptyset$ for $i\neq j$ is transparent. Hence $U$ is a feasible open set for $F$.
 The generator of the associated tiling $\sT(U)$ is $U\setminus \Phi(\overline{U})=B$. Since $\bd B\subset F$, we conclude that $\sT(U)$ is compatible. Hence we have constructed a compatible tiling for $F$, which completes the proof.
\end{proof}

\paragraph{\bf Acknowledgements.} During the work on this article the authors were supported by a Czech-German cooperation grant commonly funded by GA\v{C}R and DFG, project no. GA\v{C}R P201/10/J039 and DFG WE 1613/2-1.  We are grateful to A. Kravchenko and D. Mekhontsev, the authors of the software package \emph{IFS Builder 3d}, which we have used to create the figures of the examples in the paper. We thank J. Rataj, M. Zähle and T.~Bohl for helpful comments and fruitful discussions.

\end{document}